\documentclass[11pt, a4paper, reqno]{amsart}

\usepackage[T1]{fontenc}
\usepackage[utf8]{inputenc}

\usepackage{newtxtext}
\usepackage{newtxmath}

\usepackage[scaled=0.90]{helvet}

\usepackage{courier}

\usepackage[mathcal]{eucal}

\usepackage{xcolor}

\usepackage{microtype}                   
\usepackage[
    a4paper,
    top=2.5cm,
    bottom=2.5cm,
    left=2.5cm,
    right=2.5cm,
    headheight=14pt,
    headsep=1.2em,
    footskip=2.5em,
]{geometry}

\usepackage{fancyhdr}
\fancypagestyle{plain}{%
    \fancyhf{}
    
    \fancyfoot[C]{\small\thepage}
}

\usepackage{amsmath, amsthm}
\usepackage{mathtools}
\usepackage{bm}
\usepackage{nicefrac}
\usepackage{stmaryrd}

\usepackage{aliascnt}

\newcommand{\newtheoremalias}[2]{%
  \newtheorem{#1}{#2}%
}

\theoremstyle{plain}
\newtheorem{theorem}{Theorem}
\newtheoremalias{lemma}{Lemma}
\newtheoremalias{proposition}{Proposition}
\newtheoremalias{corollary}{Corollary}

\newtheoremalias{claim}{Claim}

\theoremstyle{definition}
\newtheoremalias{definition}{Definition}
\newtheoremalias{example}{Example}
\newtheoremalias{examples}{Examples}
\newtheoremalias{exercise}{Exercise}

\newtheoremalias{question}{Question}
\newtheoremalias{setup}{Setup}
\newtheoremalias{notation}{Notation}

\theoremstyle{remark}
\newtheoremalias{remark}{Remark}
\newtheoremalias{remarks}{Remarks}
\newtheoremalias{note}{Note}
\newtheoremalias{observation}{Observation}

\newtheorem*{theorem*}{Theorem}
\newtheorem*{lemma*}{Lemma}
\newtheorem*{proposition*}{Proposition}
\newtheorem*{corollary*}{Corollary}
\newtheorem*{conjecture*}{Conjecture}
\newtheorem*{remark*}{Remark}

\AtBeginDocument{}

\usepackage{graphicx}
\usepackage{float}
\usepackage{booktabs}
\usepackage{caption}
\usepackage{tikz}
\usepackage{tikz-cd}       
\usetikzlibrary{arrows.meta, calc, positioning, decorations.pathmorphing,
                 decorations.pathreplacing,
                 shapes.geometric, patterns, shadows}

\usepackage[
    colorlinks  = true,
    linkcolor   = cyan!80!black,
    citecolor   = purple,
    urlcolor    = blue!60!black,
    backref     = page,
]{hyperref}
\usepackage[capitalize, noabbrev]{cleveref}

\crefname{subsection}{Subsection}{Subsections}
\Crefname{subsection}{Subsection}{Subsections}

\definecolor{secref}{HTML}{B22222}
\definecolor{figref}{HTML}{0D3B66}
\definecolor{thmref}{rgb}{0, 0.7, 0.7}
\AtBeginDocument{%
  \crefformat{section}{\textcolor{secref}{Section~}#2\textcolor{secref}{#1}#3}%
  \Crefformat{section}{\textcolor{secref}{Section~}#2\textcolor{secref}{#1}#3}%
  \crefformat{subsection}{\textcolor{secref}{Subsection~}#2\textcolor{secref}{#1}#3}%
  \Crefformat{subsection}{\textcolor{secref}{Subsection~}#2\textcolor{secref}{#1}#3}%
  \crefformat{figure}{\textcolor{figref}{Figure~}#2\textcolor{figref}{#1}#3}%
  \Crefformat{figure}{\textcolor{figref}{Figure~}#2\textcolor{figref}{#1}#3}%
  \crefformat{theorem}{\textcolor{thmref}{Theorem~}#2\textcolor{thmref}{#1}#3}%
  \Crefformat{theorem}{\textcolor{thmref}{Theorem~}#2\textcolor{thmref}{#1}#3}%
  \crefformat{lemma}{\textcolor{thmref}{Lemma~}#2\textcolor{thmref}{#1}#3}%
  \Crefformat{lemma}{\textcolor{thmref}{Lemma~}#2\textcolor{thmref}{#1}#3}%
  \crefformat{proposition}{\textcolor{thmref}{Proposition~}#2\textcolor{thmref}{#1}#3}%
  \Crefformat{proposition}{\textcolor{thmref}{Proposition~}#2\textcolor{thmref}{#1}#3}%
  \crefformat{corollary}{\textcolor{thmref}{Corollary~}#2\textcolor{thmref}{#1}#3}%
  \Crefformat{corollary}{\textcolor{thmref}{Corollary~}#2\textcolor{thmref}{#1}#3}%
  \crefformat{definition}{\textcolor{thmref}{Definition~}#2\textcolor{thmref}{#1}#3}%
  \Crefformat{definition}{\textcolor{thmref}{Definition~}#2\textcolor{thmref}{#1}#3}%
  \crefformat{example}{\textcolor{thmref}{Example~}#2\textcolor{thmref}{#1}#3}%
  \Crefformat{example}{\textcolor{thmref}{Example~}#2\textcolor{thmref}{#1}#3}%
  \crefformat{remark}{\textcolor{thmref}{Remark~}#2\textcolor{thmref}{#1}#3}%
  \Crefformat{remark}{\textcolor{thmref}{Remark~}#2\textcolor{thmref}{#1}#3}%
  \crefformat{claim}{\textcolor{thmref}{Claim~}#2\textcolor{thmref}{#1}#3}%
  \Crefformat{claim}{\textcolor{thmref}{Claim~}#2\textcolor{thmref}{#1}#3}%
  \crefformat{conjecture}{\textcolor{thmref}{Conjecture~}#2\textcolor{thmref}{#1}#3}%
  \Crefformat{conjecture}{\textcolor{thmref}{Conjecture~}#2\textcolor{thmref}{#1}#3}%
  \crefformat{problem}{\textcolor{thmref}{Problem~}#2\textcolor{thmref}{#1}#3}%
  \Crefformat{problem}{\textcolor{thmref}{Problem~}#2\textcolor{thmref}{#1}#3}%
}

\crefname{theorem}{Theorem}{Theorems}
\Crefname{theorem}{Theorem}{Theorems}
\crefname{lemma}{Lemma}{Lemmas}
\Crefname{lemma}{Lemma}{Lemmas}
\crefname{proposition}{Proposition}{Propositions}
\Crefname{proposition}{Proposition}{Propositions}
\crefname{corollary}{Corollary}{Corollaries}
\Crefname{corollary}{Corollary}{Corollaries}
\crefname{definition}{Definition}{Definitions}
\Crefname{definition}{Definition}{Definitions}
\crefname{example}{Example}{Examples}
\Crefname{example}{Example}{Examples}
\crefname{remark}{Remark}{Remarks}
\Crefname{remark}{Remark}{Remarks}
\crefname{claim}{Claim}{Claims}
\Crefname{claim}{Claim}{Claims}
\crefname{conjecture}{Conjecture}{Conjectures}
\Crefname{conjecture}{Conjecture}{Conjectures}
\crefname{problem}{Problem}{Problems}
\Crefname{problem}{Problem}{Problems}

\makeatletter
\renewcommand{\@cite}[2]{%
    {\scriptsize\raise0.5pt\hbox{[{#1\if@tempswa , #2\fi}]}}%
}
\makeatother

\definecolor{backrefcolor}{HTML}{8B4513}
\definecolor{citpagecolor}{HTML}{D2691E}

\makeatletter
\renewcommand*{\backref}[1]{}

\renewcommand*{\backrefalt}[4]{%
    \ifcase #1 %
    \or {\hspace{0.8em}\footnotesize\color{black}(cit.\ on p.~#2).}%
    \else {\hspace{0.8em}\footnotesize\color{black}(cit.\ on pp.~#2).}%
    \fi
}

\makeatother

\definecolor{emailcolor}{HTML}{E6007E}
\makeatletter
\let\orig@email\email
\renewcommand{\email}[1]{\orig@email{\color{emailcolor}#1}}
\makeatother

\makeatletter
\renewcommand{\section}{\@startsection{section}{1}%
    {\z@}{-2.5ex \@plus -1ex \@minus -.2ex}{1.5ex \@plus .2ex}%
    {\normalfont\fontsize{16}{19}\bfseries}}
\renewcommand{\subsection}{\@startsection{subsection}{2}%
    {\z@}{-2ex \@plus -0.8ex \@minus -.2ex}{1ex \@plus .2ex}%
    {\normalfont\fontsize{13}{16}\bfseries}}
\renewcommand{\subsubsection}{\@startsection{subsubsection}{3}%
    {\z@}{-1.5ex \@plus -0.5ex \@minus -.1ex}{0.8ex \@plus .1ex}%
    {\normalfont\fontsize{11}{14}\bfseries\itshape}}

\renewcommand{\@settitle}{\begin{center}%
    \baselineskip20pt\relax
    \bfseries\fontsize{18}{22}\selectfont
    \@title
    \end{center}%
}
\let\uppercasenonmath\@gobble
\AtBeginDocument{}

\renewcommand{\@setauthors}{%
    \begingroup
    \trivlist
    \centering
    \fontsize{13}{16}\selectfont
    \@topsep30\p@\relax
    \advance\@topsep by -\baselineskip
    \item\relax
    \andify\authors
    \def\\{\protect\linebreak}%
    \authors
    \endtrivlist
    \endgroup
}
\makeatother

\definecolor{toccolor}{RGB}{27,58,92}
\definecolor{wblue}{HTML}{2D6CA2}   
\definecolor{worange}{HTML}{E26B43} 
\definecolor{ydbdr}{HTML}{5C3A21}    
\definecolor{ydwarm1}{HTML}{E8A87C}  
\definecolor{ydwarm2}{HTML}{F2D5A8}  
\definecolor{ydcore}{HTML}{D17A47}   
\definecolor{tbl-bbb}{HTML}{8C3A60}   
\definecolor{tbl-bba}{HTML}{2E3F66}   
\definecolor{tbl-bab}{HTML}{1F6E73}   
\definecolor{tbl-abb}{HTML}{2E6B3A}   
\definecolor{tbl-baa}{HTML}{3D7BBF}   
\definecolor{tbl-aba}{HTML}{B58A1D}   
\definecolor{tbl-aab}{HTML}{A33B6B}   
\definecolor{tbl-aaa}{HTML}{7A6E2A}   

\makeatletter
\renewcommand{\tableofcontents}{%
    \par
    \begin{list}{}{%
        \leftmargin2.5pc \rightmargin2.5pc
        \listparindent\z@ \itemindent\z@
        \parsep\z@ \topsep6pt \partopsep\z@}
    \item\relax
    {\large\bfseries Contents}\par
    \vspace{4pt}%
    \@input{\jobname.toc}%
    \if@filesw
        \expandafter\newwrite\csname tf@toc\endcsname
        \immediate\openout \csname tf@toc\endcsname \jobname.toc\relax
    \fi
    \global\@nobreakfalse
    \end{list}%
}
\makeatother

\makeatletter
\renewcommand{\tocsection}[3]{%
    \indentlabel{\@ifnotempty{#2}{\color{toccolor}\ignorespaces#1 #2.\quad}}{\bfseries\color{toccolor}#3}}
\renewcommand{\tocsubsection}[3]{%
    \indentlabel{\@ifnotempty{#2}{\color{toccolor}\ignorespaces#1 #2.\quad}}{\color{toccolor}#3}}
\renewcommand{\tocappendix}[3]{%
    \indentlabel{\@ifnotempty{#2}{\color{toccolor}\ignorespaces#1 #2.\quad}}{\bfseries\color{toccolor}#3}}

\def\l@section{\@tocline{1}{2pt}{0pt}{}{}}
\def\l@subsection{\@tocline{2}{0pt}{1.5em}{}{}}

\def\@tocline#1#2#3#4#5#6#7{\relax
    \ifnum #1>\c@tocdepth
    \else
        \par \addpenalty\@secpenalty\addvspace{#2}%
        \begingroup \hyphenpenalty\@M \small
        \@ifempty{#4}{%
            \@tempdima\csname r@tocindent\number#1\endcsname\relax
        }{%
            \@tempdima#4\relax
        }%
        \parindent\z@ \leftskip#3\relax \advance\leftskip\@tempdima\relax
        \rightskip\@pnumwidth plus4em \parfillskip-\@pnumwidth
        #5\leavevmode\hskip-\@tempdima
        {\color{toccolor}#6}\nobreak\relax
        \leaders\hbox{$\color{toccolor}\m@th
            \mkern 4.5mu\hbox{\normalfont\small\color{toccolor}.}%
            \mkern 4.5mu$}\hfill
        \hbox to\@pnumwidth{\@tocpagenum{\color{toccolor}#7}}\par
        \nobreak
        \endgroup
    \fi}
\makeatother

\usepackage{enumitem}
\setlist{leftmargin=*, itemsep=3pt, parsep=1pt}
\setlist[enumerate,1]{label=(\roman*)}
\setlist[enumerate,2]{label=(\alph*)}

\newcommand{\N}{\mathbb{N}}
\newcommand{\Z}{\mathbb{Z}}

\newcommand{\R}{\mathbb{R}}

\DeclareMathOperator{\sgn}{sgn}

\newcommand{\eps}{\varepsilon}

\newcommand{\defeq}{\coloneqq}

\newcommand{\intv}[2]{\llbracket #1,#2\rrbracket}   

\newcommand{\tta}{\text{\tt a}}
\newcommand{\ttb}{\text{\tt b}}
\newcommand{\wrd}{\mathrm{w}}
\newcommand{\qv}{\mathrm{q}}        
\newcommand{\wu}{\mathrm{u}}        
\newcommand{\wv}{\mathrm{v}}        
\newcommand{\A}{\mathcal{A}}        
\definecolor{seqteal}{HTML}{1F6E73}    
\newcommand{\seqnum}[1]{\href{https://oeis.org/#1}{\textcolor{seqteal}{\textsf{\underline{#1}}}}}

\definecolor{abstractbg}{gray}{0.93}

\makeatletter
\renewenvironment{abstract}{%
    \ifx\maketitle\relax
        \ClassWarning{\@classname}{Abstract should precede
            \protect\maketitle}%
    \fi
    \global\setbox\abstractbox=\vtop\bgroup
        \normalfont\small
        \list{}{\labelwidth\z@
            \leftmargin5pc \rightmargin\leftmargin
            \listparindent\normalparindent \itemindent\z@
            \parsep\z@ \@plus\p@
            
        }%
        \item\relax
        \begin{center}\textbf{Abstract}\end{center}%
        \vspace{-2pt}%
        \noindent\ignorespaces
}{%
    \endlist\egroup
    \ifx\@setabstract\relax \@setabstracta \fi
}

\let\orig@setabstracta\@setabstracta
\renewcommand{\@setabstracta}{%
    \orig@setabstracta
}
\makeatother

\title{Binary binomial equivalence via hyperplane arrangements}

\author{Mehdi Golafshan}
\thanks{Supported by the FNRS Research grant T.196.23 (PDR)}
\thanks{The author is grateful to Michel Rigo for valuable comments and
suggestions on an earlier version of this paper}
\address{Department of Mathematics, University of Li\`{e}ge, Li\`{e}ge, Belgium}
\email{mgolafshan@uliege.be}

\date{}

\begin{document}

\begin{abstract}
Rigo and Salimov (2015) proved that the number of binary \(2\)-binomial
equivalence classes of words of length \(n\) is the \(n^{\text{th}}\) cake
number.  We give a geometric explanation of this identity by constructing an
explicit arrangement of \(n\) planes in three-dimensional space whose chambers
are naturally indexed by these equivalence classes.

This arrangement is the three-dimensional member of an infinite family of
hyperplane arrangements.  In dimensions \(1\), \(2\), and \(3\), the
corresponding quotients recover abelian equivalence, a natural intermediate
equivalence between abelian and \(2\)-binomial equivalence, and binary
\(2\)-binomial equivalence itself.  In higher dimensions, the same family
realises natural refinements of \(2\)-binomial equivalence.

We also determine the sizes of the resulting classes.  Each size is given by a
coefficient of a suitable Gaussian binomial coefficient.  This yields the full
class-size distribution for binary \(2\)-binomial equivalence and stabilisation
results for the number of classes of any fixed cardinality.
\end{abstract}

\maketitle

\vspace{-4pt}
\begin{list}{}{%
    \leftmargin3.5pc \rightmargin3.5pc
    \listparindent0pt \itemindent0pt
    \parsep4pt \topsep0pt \partopsep0pt}
\item\relax

\item\relax
\end{list}
\vspace{4pt}

\section{Introduction}\label{sec:intro}

The \emph{binomial coefficient} of a word $\wu$ over a word $\wv$
counts the occurrences of $\wv$ as a scattered subword of $\wu$.  Building on
this notion, Rigo and Salimov~\cite{RS} introduced the \emph{$k$-binomial
equivalences} $\sim_k$ on words over a finite alphabet: two words are
$k$-binomially equivalent if every word of length at most $k$ occurs as a
scattered subword of each with the same multiplicity.  The relation
$\sim_1$ is abelian equivalence; larger $k$ gives strictly finer partitions,
all coarser than equality.  The binary fixed-length relations used in this
paper are fixed in \cref{subsec:main-relations}.

The relation $\sim_k$ has been studied from several angles.  Lejeune et
al.~\cite{LejeuneRigoRosenfeld2020} studied the structure of
$2$-binomial classes and proved that the number of
$\sim_2$-classes of length-$n$ words over an $m$-letter alphabet grows as
$\Theta(n^{m^2-1})$; Chrisnata et al.~\cite{Chrisnata2023} sharpened the
asymptotic upper bound for $k \geq 3$; and Freydenberger et
al.~\cite{FreydenbergerEtAl2015} gave polynomial-time algorithms
deciding $\sim_k$.  In a different direction, Golafshan and
Rigo~\cite{GolafRigo1, GolafRigo2} extended the binomial coefficient to
two-dimensional rectangular arrays.

\medskip
\noindent\textbf{The motivating identity.}\enspace
A basic non-trivial case for $\sim_k$ is the binary
alphabet at $k = 2$.  Rigo and Salimov~\cite[Lemma~4]{RS} established
\begin{equation}\label{eq:cake_identity}
\#\bigl(\{\tta, \ttb\}^{n}\,/\,\sim_{2}\bigr)
\;=\;
\frac{n^{3}+5n+6}{6}
\;=\;
\sum_{r=0}^{3}\binom{n}{r}.
\end{equation}
The right-hand side is the \(n^{\text{th}}\) \emph{cake number}~(\seqnum{A000125}): the
number of regions into which $n$ planes in general position divide
$\R^{3}$.  More generally, the Steiner--Schl\"afli formula counts the
chambers of $n$ hyperplanes in general position in $\R^{d}$ as
$\sum_{r=0}^{d}\binom{n}{r}$~\cite{Sch}; see~\cite{AW, Sta2} for historical
and modern accounts.

\medskip
Thus \eqref{eq:cake_identity} shows that an algebraic enumeration of
equivalence classes coincides with the chamber count of a generic arrangement.
This numerical coincidence suggests a geometric model for the quotient
$\{\tta,\ttb\}^{n}/{\sim_{2}}$.  Richomme~\cite{Ric} gave a geometric interpretation of certain binary
$2$-binomial coefficients and described a lattice structure on the
representatives of a fixed $\sim_{2}$-class.  The present point of view is
complementary: instead of studying the geometry inside one class, we construct
one hyperplane arrangement whose chambers are in bijection with the
quotient $\{\tta,\ttb\}^{n}/{\sim_{2}}$.

The work of Lejeune et al.~\cite{LejeuneRigoRosenfeld2020} gives an algebraic
model for the same quotient: $\A^{*}/{\sim_2}$ identifies with the positive
submonoid of the free nilpotent group of class $2$.  In the binary case this
description reduces to the complete invariant
\begin{equation}\label{eq:binary-complete-invariant}
\left(|\wrd|_{\tta},\binom{\wrd}{\tta\ttb}\right),
\end{equation}
the binary fixed-length invariant used throughout the paper.  Our construction is
complementary: it realises this binary quotient as the chamber set of a concrete
hyperplane arrangement.

\medskip\noindent\textbf{Notation.}\enspace
We denote by \(\N:=\{0,1,2,\ldots\}\) the set of non-negative integers, and set
\(\N_{>0}:=\N\setminus\{0\}\).  For \(i,j\in\Z\) with \(i\leq j\), we write
\(\intv{i}{j}:=\{i,i+1,\ldots,j\}\) for the integer interval from \(i\) to \(j\).

\subsection{Main equivalence relations}\label{subsec:main-relations}

Let \(\A\) be a finite alphabet and \(\A^{*}\) the set of finite words over
\(\A\).  For \(\wrd\in\A^{*}\), write \(|\wrd|\) for the
length and, for \(\tta\in\A\), \(|\wrd|_{\tta}\) for the number of occurrences
of \(\tta\) in \(\wrd\).  Given \(\wrd=\wrd_1\cdots\wrd_m\) and \(\wv\) in
\(\A^{*}\), the binomial coefficient
\[
\binom{\wrd}{\wv}:=\#\{1\le i_1<\cdots<i_{|\wv|}\le m :
\wrd_{i_1}\cdots\wrd_{i_{|\wv|}}=\wv\}.
\]
We adopt the conventions \(\binom{\wrd}{\eps}=1\) (\(\eps\) the empty word) and \(\binom{\wrd}{\wv}=0\) for
\(|\wv|>|\wrd|\) (see~\cite[Section~6.3]{Lot}).

\begin{definition}[{\cite[Definition~2]{RS}}]\label{def:k-binomial}
Let \(k\in\N_{>0}\cup\{\infty\}\) and \(\wu,\wrd\in\A^{*}\).  We say \(\wu\)
and \(\wrd\) are \emph{\(k\)-binomially equivalent} if
\(\binom{\wu}{\wv}=\binom{\wrd}{\wv}\) for every \(\wv\in\A^{*}\) with
\(|\wv|\le k\).
\end{definition}

In this paper, \(k\)-binomial equivalence is denoted by \(\wrd\sim_k\wu\).
Since \(\binom{\wu}{\tta}=|\wu|_{\tta}\) for every \(\tta\in\A\), the relation
\(\sim_1\) is abelian equivalence; and \(\sim_{k+1}\subseteq\sim_k\) for every
\(k\in\N_{>0}\).

We focus on the binary alphabet \(\A=\{\tta,\ttb\}\).  Let \(\wrd,\wu\in\{\tta,\ttb\}^n\); then we have
\begin{itemize}
\item \(\wrd\sim_1\wu \iff |\wrd|_{\tta}=|\wu|_{\tta}\);
\item \(\wrd\sim_2\wu \iff |\wrd|_{\tta}=|\wu|_{\tta}\) and
\(\binom{\wrd}{\tta\ttb}=\binom{\wu}{\tta\ttb}\).
\end{itemize}

\begin{example}\label{ex:abelian-vs-2-binom}
Let \(w_1=\tta\ttb\ttb\tta\), \(w_2=\ttb\tta\tta\ttb\), and
\(w_3=\tta\tta\ttb\ttb\).  They are pairwise abelian
equivalent; \(w_1\sim_2 w_2\), while \(w_3\) is
\(2\)-binomially equivalent to neither.
\end{example}

Set
\[
\operatorname{strip}(\wrd):=\begin{cases}
\Bigl\lfloor\binom{\wrd}{\tta\ttb}\bigm/|\wrd|_{\tta}\Bigr\rfloor & \text{if } |\wrd|_{\tta}>0, \\[2pt]
0 & \text{if } |\wrd|_{\tta}=0.
\end{cases}
\]

\begin{definition}\label{def:three-halves}\label{def:sim-32}
Let \(\wu,\wrd\in\{\tta,\ttb\}^*\).  We say that \(\wu\) and \(\wrd\) are
\emph{\(3/2\)-binomially equivalent} if
\[
   \wu\sim_1\wrd
   \qquad\text{and}\qquad
   \operatorname{strip}(\wu)=\operatorname{strip}(\wrd).
\]
We denote this relation by \(\wu\sim_{3/2}\wrd\).
\end{definition}

\begin{remark}
The notation \(\sim_{3/2}\) does not extend \(k\)-binomial equivalence to
non-integer \(k\).  It indicates that \(\sim_{3/2}\) lies strictly between the
first two integer levels of the \(k\)-binomial hierarchy:
\(\sim_2 \subsetneq \sim_{3/2} \subsetneq \sim_1\).
\end{remark}

\begin{example}\label{ex:three-halves}
Let \(w_1=\tta\ttb\tta\ttb\), \(w_2=\tta\ttb\ttb\tta\), and
\(w_3=\tta\tta\ttb\ttb\).  They are pairwise abelian equivalent;
\(w_1\sim_{3/2} w_2\) but \(w_1\not\sim_2 w_2\), while \(w_3\) is
\(3/2\)-binomially equivalent to neither, because
\(\operatorname{strip}(w_1)=\operatorname{strip}(w_2)=1\) and \(\operatorname{strip}(w_3)=2\).
\end{example}

Associate to each \(\wrd\in\{\tta,\ttb\}^n\) the Young diagram
\(\lambda(\wrd):=(r_1,\ldots,r_{|\wrd|_{\tta}})\), where \(r_i\) is the number of
\(\ttb\)'s to the right of the \(i^{\text{th}}\) \(\tta\) in \(\wrd\); then
\(r_1\ge\cdots\ge r_{|\wrd|_{\tta}}\ge 0\) and
\(|\lambda(\wrd)|=\binom{\wrd}{\tta\ttb}\).  Explicitly, writing
\[
\wrd=\ttb^{i_0}\tta\ttb^{i_1}\tta\ttb^{i_2}\tta\cdots\tta\ttb^{i_{|\wrd|_{\tta}}},
\]
we have \(r_j=i_j+i_{j+1}+\cdots+i_{|\wrd|_{\tta}}\).

\begin{definition}\label{def:s-tail}
Let \(\wrd\in\{\tta,\ttb\}^n\), \(s\in\N\), and \(\ell:=\min(s,|\wrd|_{\tta})\).
The \emph{\(s\)-tail} of \(\wrd\) is
\(\mathcal{T}_s(\wrd):=(r_{|\wrd|_{\tta}},\ldots,r_{|\wrd|_{\tta}-\ell+1})\),
empty when \(\ell=0\).
\end{definition}

We let \(\theta=r_{|\wrd|_{\tta}-\ell+1}\) denote the last entry of \(\mathcal{T}_s(\wrd)\)
and \(|\mathcal{T}_s(\wrd)|\) its sum; both are \(0\) when \(\ell=0\).  Note that
\(\ell\) and \(\theta\) depend on \(s\), a dependence left implicit in the notation.

\begin{example}\label{ex:s-tail}
Let \(\wrd=\tta\ttb\ttb\tta\ttb\), so \(\lambda(\wrd)=(3,1)\).  Then
\(\mathcal{T}_0(\wrd)=\varnothing\), \(\mathcal{T}_1(\wrd)=(1)\), and
\(\mathcal{T}_s(\wrd)=(1,3)\) for \(s\ge 2\).
\end{example}

The following equivalence relations provide the word-combinatorial
classification of chambers in dimension greater than three.

\begin{definition}\label{def:approx-s}
Let \(s\in\N\) and \(\wu,\wrd\in\{\tta,\ttb\}^n\).  We say that \(\wu\) and
\(\wrd\) are \emph{\(s\)-quasi-binomially equivalent} if \(\wu\sim_2\wrd\) and
\(\mathcal{T}_s(\wu)=\mathcal{T}_s(\wrd)\).
\end{definition}

We write \(\wu\approx_s\wrd\) for \(s\)-quasi-binomial equivalence, and
\([\wrd]_s\) for the \(\approx_s\)-class of \(\wrd\).
By construction \(\approx_0=\sim_2\) and \(\approx_{s+1}\subseteq\approx_s\).

\begin{example}\label{ex:approx-s}
Let \(\wu=\tta\ttb\ttb\ttb\tta\tta\) and \(\wrd=\ttb\tta\ttb\tta\ttb\tta\),
with \(\lambda(\wu)=(3,0,0)\) and \(\lambda(\wrd)=(2,1,0)\).  Then
\(\wu\sim_2\wrd\), they share the \(1\)-tail \((0)\), but their \(2\)-tails
\((0,0)\) and \((0,1)\) differ; hence \(\wu\approx_1\wrd\) but
\(\wu\not\approx_2\wrd\).
\end{example}

\subsection{Main results and contributions}

For \(d\in\N_{>0}\), let \(\mathcal{H}_n^{(d)}\) be the arrangement in \(\R^d\)
whose hyperplanes are
\[
H_i^{(d)}
=
\{\mathbf{x}=(x_1,\ldots,x_d)\in\R^d:
1+x_1i+x_2i^2+\cdots+x_di^d=0\}
\]
for \(i\in\intv{1}{n}\).
The paper has two aims: to realise the binomial-equivalence quotients of
\(\{\tta,\ttb\}^n\) as the chambers of \(\mathcal{H}_n^{(d)}\), and to enumerate the
resulting classes.

\medskip\noindent\textbf{A chamber model for binomial equivalence.}\enspace
The key case is dimension three: the chambers of
\(\mathcal{H}_n^{(3)}\) are in bijection with the \(\sim_2\)-classes of
\(\{\tta,\ttb\}^n\), recasting the cake-number identity~\eqref{eq:cake_identity} as a
chamber count.  Higher dimensions are matched by the \(s\)-quasi-binomial relations
\(\approx_s\), with \(\approx_0=\sim_2\), describing the chambers for \(d\geq 4\); the
two lowest cases recover abelian equivalence \(\sim_1\) at \(d=1\) and the intermediate
relation \(\sim_{3/2}\) at \(d=2\), the latter lying strictly between \(\sim_1\) and
\(\sim_2\).  These cases combine into a single chamber-to-class bijection, valid for
every \(d\in\N_{>0}\):
\[
\operatorname{Ch}(\mathcal{H}_n^{(d)})
\xleftrightarrow{\;\Psi_{d,n}\;}
\begin{cases}
\{\tta,\ttb\}^n/\sim_1, & \text{if } d=1,\\[1mm]
\{\tta,\ttb\}^n/\sim_{3/2}, & \text{if } d=2,\\[1mm]
\{\tta,\ttb\}^n/\sim_2, & \text{if } d=3,\\[1mm]
\{\tta,\ttb\}^n/\!\approx_{d-3}, & \text{if } d\geq 4.
\end{cases}
\]
This bijection is proved in \cref{sec:higher-tail} for \(d\ge 3\) and in
\cref{sec:lower-dim} for \(d=1,2\), using the general position and chamber sign words
established in \cref{sec:polyarr}.  The arrangement therefore has
\(\sum_{r=0}^{d}\binom{n}{r}\) chambers, the entry \(B(n,d)\) of Bernoulli's triangle;
its first four columns realise the relations \(\sim_1,\sim_{3/2},\sim_2,\approx_1\);
see \cref{fig:bernoulli-triangle}.

\begin{figure}[H]
\centering
\definecolor{btC1}{HTML}{6A2C70}   
\definecolor{btC2}{HTML}{B83B5E}   
\definecolor{btC3}{HTML}{F08A5D}   
\definecolor{btC4}{HTML}{F9B248}   
\definecolor{btL3}{HTML}{CC5A2E}   
\definecolor{btL4}{HTML}{B8770E}   
\definecolor{btBg}{HTML}{E8ECF1}   
\definecolor{btInk}{HTML}{334155}  
\definecolor{btDk}{HTML}{243044}   
\definecolor{btMute}{HTML}{64748B} 
\begin{tikzpicture}[
    x=5.1em, y=2.6em,
    cellbase/.style={minimum width=3.5em, minimum height=2.4em,
                     rounded corners=4pt, inner sep=2pt,
                     font=\normalsize, draw=none},
    cA/.style={cellbase, fill=btC1, text=white},
    cB/.style={cellbase, fill=btC2, text=white},
    cC/.style={cellbase, fill=btC3, text=btDk},
    cD/.style={cellbase, fill=btC4, text=btDk},
    c0/.style={cellbase, fill=btBg, text=btInk},
    nlab/.style={font=\small, anchor=east, text=btMute},
    dlab/.style={font=\small, anchor=south, text=btInk},
    seqname/.style={font=\footnotesize, anchor=north, text=btMute,
                    align=center, inner sep=1pt},
    seqcode/.style={font=\footnotesize, anchor=north,
                    align=center, inner sep=2pt},
]
\foreach \n/\k/\val/\st in {%
    0/0/1/c0,
    1/0/1/c0, 1/1/2/cA,
    2/0/1/c0, 2/1/3/cA, 2/2/4/cB,
    3/0/1/c0, 3/1/4/cA, 3/2/7/cB, 3/3/8/cC,
    4/0/1/c0, 4/1/5/cA, 4/2/11/cB, 4/3/15/cC, 4/4/16/cD,
    5/0/1/c0, 5/1/6/cA, 5/2/16/cB, 5/3/26/cC, 5/4/31/cD, 5/5/32/c0,
    6/0/1/c0, 6/1/7/cA, 6/2/22/cB, 6/3/42/cC, 6/4/57/cD, 6/5/63/c0, 6/6/64/c0%
} \node[\st] at (\k, -\n) {\val};
\foreach \n in {0,...,6}
  \node[nlab] at (-0.5, -\n) {$n=\n$};
\node[font=\small, anchor=south, text=btC1] at (1,1.40){$\sim_1$};
\node[font=\small, text=btInk, rotate=90] at (1,0.92){$d=1$};
\node[font=\small, anchor=south, text=btC2] at (2,1.40){$\sim_{3/2}$};
\node[font=\small, text=btInk, rotate=90] at (2,0.92){$d=2$};
\node[font=\small, anchor=south, text=btL3] at (3,1.40){$\sim_2$};
\node[font=\small, text=btInk, rotate=90] at (3,0.92){$d=3$};
\node[font=\small, anchor=south, text=btL4] at (4,1.40){$\approx_1$};
\node[font=\small, text=btInk, rotate=90] at (4,0.92){$d=4$};
\node[seqname] at (1, -6.66) {line-cutting};
\node[seqcode] at (1, -7.10) {\seqnum{A000027}};
\node[seqname] at (2, -6.66) {lazy caterer};
\node[seqcode] at (2, -7.10) {\seqnum{A000124}};
\node[seqname] at (3, -6.66) {cake numbers};
\node[seqcode] at (3, -7.10) {\seqnum{A000125}};
\node[seqname] at (4, -6.66) {circle-cutting};
\node[seqcode] at (4, -7.10) {\seqnum{A000127}};
\end{tikzpicture}
\caption{Bernoulli's triangle, with the four columns \(d = 1, 2, 3, 4\)
highlighted.}
\label{fig:bernoulli-triangle}
\end{figure}
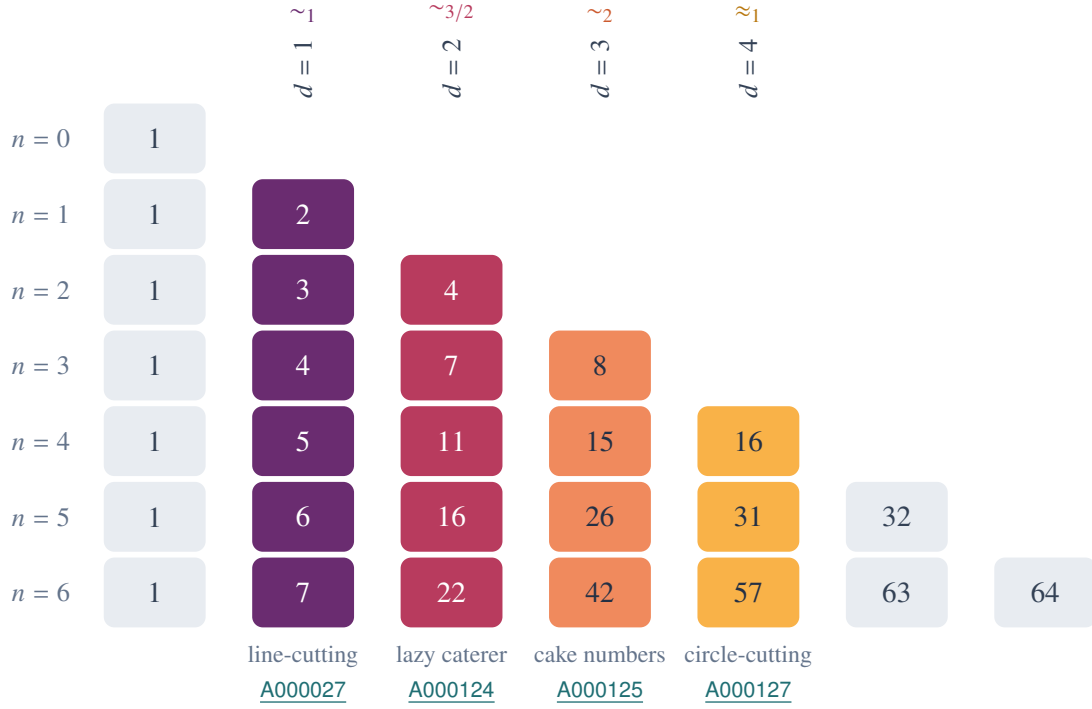

\medskip\noindent\textbf{Enumeration of the classes.}\enspace
\Cref{sec:class-size} computes the class sizes.  The size of each \(\approx_s\)-class
is a single coefficient of a Gaussian binomial
coefficient, the \(\sim_2\)- and \(\sim_{3/2}\)-counts following as the case \(s=0\)
and as a coarsening.  Their distribution is determined as well: for binary
\(2\)-binomial equivalence, the number of classes of a given size \(K\geq 2\) is
eventually linear in \(n\) when \(K\) is a partition number and eventually constant
otherwise; an analogous stabilisation holds for \(\sim_{3/2}\), and the singleton
\(\approx_s\)-classes are counted in closed form.

\subsection{Relation to prior work}

Binomial equivalence differs from \(k\)-abelian equivalence: both have abelian
equivalence as their first level, but \(k\)-abelian equivalence counts contiguous
factors, whereas \(k\)-binomial equivalence counts scattered subwords.  The sizes of
\(k\)-abelian classes were determined by Karhum\"{a}ki et al.~\cite{KPRW17}; the present paper counts the binomial classes instead.

The paper draws on two classical ingredients.  First, the arrangement
\(\mathcal{H}_{n}^{(d)}\) is dual to the points \((i,i^{2},\ldots,i^{d})\)
for \(i\in\intv{1}{n}\), and its chamber sign vectors are the positive-tope restrictions of the
alternating (or cyclic) oriented matroid~\cite[Section~9.4]{BLSWZ}.  The run
description used here is therefore a word-combinatorial form of a classical
sign-variation phenomenon, familiar from total positivity~\cite{Kar} and from Gale's
evenness condition for cyclic polytopes~\cite{Gal}; the novelty is not the phenomenon
but its use to model binary binomial quotients by chambers.

Second, the counting rests on a partition identity: the coefficient of \(\qv^{N}\) in
the Gaussian binomial coefficient \(\binom{r+c}{r}_{\!\qv}\) is the number of
partitions of \(N\) fitting in an \(r\times c\) rectangle (see Andrews~\cite{And98} or
Stanley~\cite[Section~1.8]{Sta12}), with asymptotics refined by Melczer et al.~\cite{MPP20}.  Here the Young-diagram encoding sends the free part of an
\(\approx_{s}\)-class to exactly such a bounded partition, and the residual area
selects the required coefficient, again turning a classical identity into exact
class-size formulas.


\section{The arrangement \texorpdfstring{\(\mathcal{H}_{n}^{(d)}\)}{Hn(d)} and chamber words}\label{sec:polyarr}

Let \(d, n \in \N_{>0}\).
A \emph{hyperplane arrangement} in $\R^{d}$ is a finite collection
$\mathcal{H} = \{H_{1}, \ldots, H_{n}\}$ of affine hyperplanes in $\R^{d}$.
A \emph{chamber} of $\mathcal{H}$ is a connected component of the complement
$\R^{d} \setminus \bigcup_{i=1}^{n} H_{i}$.
We write $\operatorname{Ch}(\mathcal{H})$ for the set of chambers of
$\mathcal{H}$.
For instance, three lines in general position in $\R^{2}$ form an arrangement
with seven chambers: one bounded triangle and six unbounded regions; see
\cref{fig:pizza-d2}.

If each $H_{i}$ is defined by the equation $L_{i}(\mathbf{x}) = 0$ for some
affine form $L_{i} \colon \R^{d} \to \R$, the \emph{sign vector} of a
point $\mathbf{x} \in \R^{d} \setminus \bigcup_{i=1}^{n} H_{i}$ is
\[
    \sigma(\mathbf{x}) \defeq
    \bigl(\sgn(L_{1}(\mathbf{x})),\, \ldots,\, \sgn(L_{n}(\mathbf{x}))\bigr)
    \in \{+,-\}^{n}.
\]
Two points of $\R^{d} \setminus \bigcup_{i=1}^{n} H_{i}$ lie in the same
chamber iff they have the same sign vector; consequently, the
chambers of $\mathcal{H}$ are in bijection with the sign vectors in the
image of $\sigma$.

An arrangement $\mathcal{H}$ in $\R^{d}$ is said to be in \emph{general
position}, or \emph{generic}, if, for every $r \in \intv{1}{d}$, any $r$
distinct hyperplanes of $\mathcal{H}$ intersect in an affine subspace of
codimension~$r$, and any $d + 1$ distinct hyperplanes of $\mathcal{H}$ have
empty intersection.

We will need the following classical formula counting the chambers of a
generic arrangement.

\begin{theorem}[Steiner~\cite{Ste}; Schl\"{a}fli~\cite{Sch}]\label{thm:SS}
For any arrangement $\mathcal{H}$ of $n$ hyperplanes in general position in
$\R^{d}$,
\[
    \#\operatorname{Ch}(\mathcal{H}) \;=\; \sum_{r=0}^{d}\binom{n}{r}.
\]
\end{theorem}

Modern expositions include Zaslavsky~\cite{Zas} and
Stanley~\cite[Section~3.11]{Sta12}; see also~\cite{Sta2}.

\subsection{The arrangement \texorpdfstring{\(\mathcal{H}_{n}^{(d)}\)}{Hn(d)}}\label{subsec:arrangement-Hd}

We now give, for $d, n \in \N_{>0}$, a concrete arrangement
in general position. For $i \in \intv{1}{n}$, let
$H_{i}^{(d)} \subset \R^{d}$ be the affine hyperplane defined below,
and set $\mathcal{H}_{n}^{(d)} = \{H_{1}^{(d)}, \ldots, H_{n}^{(d)}\}$:
\begin{equation}\label{eq:Hid}
    H_{i}^{(d)} \colon\; 1 + i x_{1} + i^{2} x_{2} + \cdots + i^{d} x_{d} = 0.
\end{equation}

\begin{proposition}\label{prop:gen-position}
The arrangement $\mathcal{H}_{n}^{(d)}$ is in general position.
\end{proposition}

\begin{proof}
For each $\mathbf{x} = (x_{1}, \ldots, x_{d}) \in \R^{d}$, define the
polynomial $P_{\mathbf{x}}(t) = 1 + x_{1} t + x_{2} t^{2} + \cdots + x_{d} t^{d}$.
By~\eqref{eq:Hid}, for every $\mathbf{x} \in \R^{d}$ and every
$H_{i}^{(d)} \in \mathcal{H}_{n}^{(d)}$,
$\mathbf{x} \in H_{i}^{(d)}$ iff $P_{\mathbf{x}}(i) = 0$.
Since $P_{\mathbf{x}}(0) = 1$ and $\deg P_{\mathbf{x}} \leq d$, the polynomial
$P_{\mathbf{x}}$ is non-zero and has at most $d$ roots; hence no point of
$\R^{d}$ lies on more than $d$ of the hyperplanes $H_{i}^{(d)}$, and any
$d + 1$ of them have empty intersection.

Fix $r \in \intv{1}{d}$ and $1 \leq i_{1} < \cdots < i_{r} \leq n$. The
intersection $\bigcap_{s=1}^{r} H_{i_{s}}^{(d)}$ is the solution set in
$\R^{d}$ of the linear system
$\sum_{q=1}^{d} i_{s}^{q}\, x_{q} = -1$ for $s \in \intv{1}{r}$.
An $r \times r$ minor of its coefficient matrix is a non-zero
Vandermonde-type determinant; hence the coefficient matrix has full rank~$r$,
and the solution set is an affine subspace of $\R^{d}$ of codimension~$r$.
\end{proof}

\subsection{Chamber sign words and runs}\label{subsec:chamber-words-runs}

A \emph{run} in a word $\wrd \in \{+,-\}^{*}$ is a maximal block of consecutive
equal letters; equivalently, the number of runs of $\wrd$ is one more than the
number of positions at which $\wrd$ changes letter. For instance, the word
$\mathord{+}\mathord{+}\mathord{-}\mathord{-}\mathord{+}$ has three runs.

Throughout this subsection, sign words are read in the increasing order
\(H_1^{(d)},\ldots,H_n^{(d)}\) of the hyperplanes. Since \(P_{\mathbf{x}}(0)>0\),
we augment every chamber sign word \(\mathbf{u}\in\{+,-\}^{n}\) with a
fixed \(+\) at position \(0\); this is why the run condition below is stated for
the augmented word \(+\mathbf{u}\). To translate sign words into binary words,
we use \(-\mapsto\tta\) and \(+\mapsto\ttb\), the opposite choice merely
exchanging \(\tta\) and \(\ttb\).

For a chamber $C$ of $\mathcal{H}_{n}^{(d)}$, recall the polynomial
$P_{\mathbf{x}}$ from the proof of \cref{prop:gen-position} and define the
\emph{sign word} of $C$ by
\[
    \sigma(C) = \sigma_{1} \cdots \sigma_{n} \in \{+,-\}^{n},
    \qquad
    \sigma_{i} = \sgn P_{\mathbf{x}}(i),
\]
for any $\mathbf{x} = (x_{1}, \ldots, x_{d}) \in C$. Since
$\mathbf{x} \in H_{i}^{(d)}$ iff $P_{\mathbf{x}}(i) = 0$ and $C$ misses every
$H_{i}^{(d)}$, the value $P_{\mathbf{x}}(i)$ is non-zero on $C$; being
continuous in $\mathbf{x}$, it is of constant sign on the connected set $C$, so
$\sigma(C)$ is well defined.

\begin{proposition}\label{thm:run-bijection}
The map $C \mapsto \sigma(C)$ is a bijection between the chambers of
$\mathcal{H}_{n}^{(d)}$ and the words $\mathbf{u} \in \{+,-\}^{n}$ for which
the augmented word $+\mathbf{u}$ has at most $d + 1$ runs.
\end{proposition}

\begin{proof}
No two chambers share the same sign vector, so the map $C \mapsto \sigma(C)$
is a bijection onto its image; it remains to identify which words occur as
sign words.

\emph{Necessity.} For $\mathbf{x} \in C$ and $\mathbf{u} = \sigma(C)$, the
polynomial $P_{\mathbf{x}}$ has degree at most $d$ and satisfies
$P_{\mathbf{x}}(0) > 0$. Each sign change between consecutive terms of
$P_{\mathbf{x}}(0), P_{\mathbf{x}}(1), \ldots, P_{\mathbf{x}}(n)$ forces, by
the intermediate value theorem, a distinct real root of $P_{\mathbf{x}}$ in
the interval $(0, n)$. Since $P_{\mathbf{x}}$ has at most $d$ real roots, the augmented
word $+\mathbf{u}$ has at most $d + 1$ runs.

\emph{Sufficiency.} Suppose $+\mathbf{u}$ has $m + 1$ runs, with
$m \in \intv{0}{d}$, and let $i_{1} < \cdots < i_{m}$ be the positions where it
changes sign. Choose $\alpha_{r} \in (i_{r} - 1, i_{r})$ and set
\[
    Q(t) = \prod_{r=1}^{m} \Bigl(1 - \frac{t}{\alpha_{r}}\Bigr),
\]
with $Q \equiv 1$ when $m = 0$. Then $Q$ has degree $m \leq d$, satisfies $Q(0) = 1$,
and its only real roots are the non-integer points $\alpha_{r}$, so
$Q(i) \neq 0$ for $i \in \intv{1}{n}$. As $i$ ranges over $\intv{0}{n}$,
$\sgn Q(i)$ changes precisely when $i$ passes some $\alpha_{r}$; since
$Q(0) = 1$, this gives $\sgn Q(i) = (+\mathbf{u})_{i}$ for all $i$.
Expressing $Q(t)$ as $1 + x_{1} t + \cdots + x_{m} t^{m}$ and
setting $x_{r} = 0$ for $r > m$ yields $\mathbf{x} \in \R^{d}$ with
$P_{\mathbf{x}} = Q$, lying in a chamber whose sign word is $\mathbf{u}$.
\end{proof}

\begin{example}\label{ex:sign-word-runs}
Take \(d=3\), \(n=6\), and
\[
  P_{\mathbf{x}}(t)=\left(1-\frac{2t}{3}\right)\left(1-\frac{2t}{7}\right)\left(1-\frac{2t}{11}\right).
\]
The signs of \(P_{\mathbf{x}}\) at \(t\in\intv{0}{6}\) form the augmented word
\(+\sigma(C)={+}{+}{-}{-}{+}{+}{-}\), whose three sign changes, one per root of
\(P_{\mathbf{x}}\), give the maximal \(d+1=4\) runs.  The corresponding binary word
is \(\wrd=\textcolor{wblue}{\ttb}\textcolor{worange}{\tta}\textcolor{worange}{\tta}\textcolor{wblue}{\ttb}\textcolor{wblue}{\ttb}\textcolor{worange}{\tta}\); see \cref{fig:sign-word-runs}.

\begin{center}
\definecolor{ink}{HTML}{28323D}
\definecolor{inkS}{HTML}{707C8A}
\definecolor{bluL}{HTML}{2D6CA2}
\definecolor{bluD}{HTML}{14315E}
\definecolor{bluF}{HTML}{D9E6F2}
\definecolor{corL}{HTML}{E26B43}
\definecolor{corF}{HTML}{FBE2D4}
\begin{tikzpicture}[x=1.3cm, y=0.9cm]
\fill[bluF] plot[smooth] coordinates{(0,1.1)(0.5,0.92)(1,0.65)(1.3,0.38)(1.5,0)} -- (0,0) -- cycle;
\fill[corF] plot[smooth] coordinates{(1.5,0)(1.8,-0.42)(2,-0.55)(2.5,-0.73)(3,-0.55)(3.2,-0.42)(3.5,0)} -- cycle;
\fill[bluF] plot[smooth] coordinates{(3.5,0)(3.8,0.42)(4,0.55)(4.5,0.73)(5,0.55)(5.2,0.42)(5.5,0)} -- cycle;
\fill[corF] plot[smooth] coordinates{(5.5,0)(5.75,-0.4)(6,-0.7)} -- (6,0) -- cycle;
\draw[ink, line width=0.8pt, -{Latex[length=2.2mm]}] (-0.4,0) -- (6.8,0)
   node[right=1pt, font=\small, text=ink] {$t$};
\foreach \i in {0,...,6} \draw[inkS, line width=0.5pt] (\i,0.07)--(\i,-0.07);
\draw[bluL, line width=1.5pt] plot[smooth] coordinates{(0,1.1)(0.5,0.92)(1,0.65)(1.3,0.38)(1.5,0)};
\draw[corL, line width=1.5pt] plot[smooth] coordinates{(1.5,0)(1.8,-0.42)(2,-0.55)(2.5,-0.73)(3,-0.55)(3.2,-0.42)(3.5,0)};
\draw[bluL, line width=1.5pt] plot[smooth] coordinates{(3.5,0)(3.8,0.42)(4,0.55)(4.5,0.73)(5,0.55)(5.2,0.42)(5.5,0)};
\draw[corL, line width=1.5pt] plot[smooth] coordinates{(5.5,0)(5.75,-0.4)(6,-0.7)};
\foreach \i/\y/\c in {0/1.1/bluL,1/0.65/bluL,2/-0.55/corL,3/-0.55/corL,4/0.55/bluL,5/0.55/bluL,6/-0.7/corL}{
   \fill[white] (\i,\y) circle (2.4pt); \fill[\c] (\i,\y) circle (1.7pt);}
\foreach \r in {1.5,3.5,5.5}{\fill[white] (\r,0) circle (2.9pt);
   \draw[ink, line width=0.9pt] (\r,0) circle (2.7pt);}
\node[font=\small, text=ink] at (4.35,1.06) {$P_{\mathbf{x}}(t)$};
\def\rdys{-1.6}
\foreach \x in {1.5,3.5,5.5} \draw[inkS, line width=0.5pt] (\x,\rdys+0.38)--(\x,\rdys-0.38);
\foreach \i/\s/\fc/\dc in {0/+/bluF/bluL,1/+/bluF/bluL,2/-/corF/corL,3/-/corF/corL,4/+/bluF/bluL,5/+/bluF/bluL,6/-/corF/corL}{
   \node[rounded corners=2.5pt, minimum size=6mm, inner sep=0, draw=\dc, fill=\fc,
         line width=0.8pt, font=\small\bfseries, text=\dc] at (\i,\rdys) {$\s$};
   \node[font=\scriptsize, text=bluD] at (\i,\rdys-0.55) {\i};}
\node[rounded corners=3pt, minimum size=7.2mm, inner sep=0, draw=ink, dashed,
      line width=0.6pt] at (0,\rdys) {};
\def\rdyb{-2.6}
\foreach \i/\b/\c in {1/\ttb/bluL,2/\tta/corL,3/\tta/corL,4/\ttb/bluL,5/\ttb/bluL,6/\tta/corL}
   \node[font=\small\bfseries, text=\c] at (\i,\rdyb) {$\b$};
\end{tikzpicture}

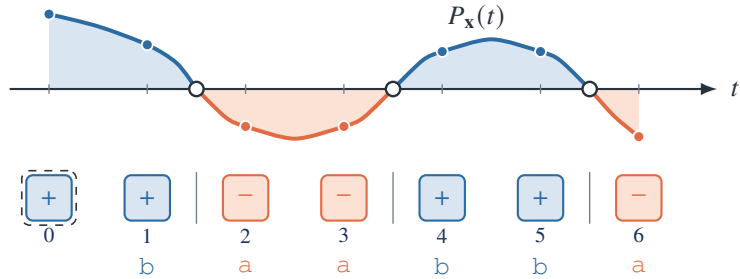
\captionof{figure}{Sign word of a chamber from a degree-\(3\) polynomial
\(P_{\mathbf{x}}\), with at most \(4\) runs.}
\label{fig:sign-word-runs}
\end{center}
\end{example}

\Cref{thm:run-bijection} makes the count of \cref{thm:SS} explicit for
$\mathcal{H}_{n}^{(d)}$. A word $\mathbf{u} \in \{+,-\}^{n}$ whose augmented
word $+\mathbf{u}$ has exactly $r + 1$ runs is determined by the $r$ positions
in $\intv{1}{n}$ at which consecutive letters of $+\mathbf{u}$ disagree, so there
are $\binom{n}{r}$ such words, and summing over $r \in \intv{0}{d}$ gives
\[
    \#\operatorname{Ch}\bigl(\mathcal{H}_{n}^{(d)}\bigr)
    = \sum_{r=0}^{d} \binom{n}{r}.
\]

\begin{remark}\label{rem:oriented-matroid}
The Vandermonde points \((1,i,i^{2},\ldots,i^{d})\), for \(i\in\intv{0}{n}\), realise
the alternating oriented matroid of rank \(\min(d+1,n+1)\) on \(n+1\) elements.
\Cref{thm:run-bijection} is the corresponding sign-variation criterion in word
form: the chamber sign words of \(\mathcal{H}_{n}^{(d)}\) are the restrictions
to the coordinates \(1,\ldots,n\) of those topes whose \(0\)-coordinate is
positive.
\end{remark}

\section{Tail refinements and the case \texorpdfstring{\(d\ge 3\)}{d >= 3}}\label{sec:higher-tail}

For \(d\geq 3\) we show that the subsets of \(\intv{1}{n}\) of size at most \(d\), which by
\cref{sec:polyarr} index the chambers of \(\mathcal{H}_n^{(d)}\), also index a natural
quotient of \(\{\tta,\ttb\}^{n}\).  Write \(s=d-3\).  We define an equivalence
relation \(\approx_s\), called \emph{\(s\)-quasi-binomial equivalence}, which
refines \(\sim_2\) by additionally recording the
last \(s\) parts of the Young diagram
\(\lambda(\wrd)\), and construct an explicit chamber-to-class bijection
\[
    \operatorname{Ch}(\mathcal{H}_n^{(d)})
    \longrightarrow
    \{\tta,\ttb\}^{n}/\approx_s.
\]
The cases \(d=1\) and \(d=2\) are treated separately in \cref{sec:lower-dim}.

\subsection{Young diagrams, tails, and residual area}
\label{subsec:young}

For binary words of fixed length \(n\), the pair
\(\bigl(|\wrd|_{\tta},\binom{\wrd}{\tta\ttb}\bigr)\)
of~\eqref{eq:binary-complete-invariant} determines all binomial
coefficients of length at most \(2\): the letter counts give
\(\binom{\wrd}{\tta\tta}=\binom{|\wrd|_{\tta}}{2}\) and
\(\binom{\wrd}{\ttb\ttb}=\binom{|\wrd|_{\ttb}}{2}\), while
\(\binom{\wrd}{\tta\ttb}+\binom{\wrd}{\ttb\tta}=|\wrd|_{\tta}\,|\wrd|_{\ttb}\)
fixes \(\binom{\wrd}{\ttb\tta}\).  It is therefore a complete invariant for
\(\sim_2\) on \(\{\tta,\ttb\}^{n}\).\label{subsec:binary-reduction}

\begin{proposition}\label{prop:complete-invariant}
The map
\(\wrd\mapsto\bigl(|\wrd|_{\tta},\binom{\wrd}{\tta\ttb}\bigr)\) induces a
bijection between \(\{\tta,\ttb\}^{n}/{\sim_2}\) and the pairs \((j,\kappa)\)
with \(0\le j\le n\) and \(0\le\kappa\le j(n-j)\).
\end{proposition}

\begin{proof}
The map is constant on \(\sim_2\)-classes and injective there by the preceding discussion; for
surjectivity, fix \(j\) with \(0\le j\le n\): the word \(\ttb^{\,n-j}\tta^j\)
has \(\binom{\wrd}{\tta\ttb}=0\) and \(\tta^j\ttb^{\,n-j}\) has
\(\binom{\wrd}{\tta\ttb}=j(n-j)\), while moving one \(\tta\) one step to the
left across one \(\ttb\) raises \(\binom{\wrd}{\tta\ttb}\) by \(1\) and fixes
\(|\wrd|_{\tta}\); hence every \(\kappa\) with \(0\le\kappa\le j(n-j)\) occurs.
\end{proof}

Recall from \cref{subsec:main-relations} the Young diagram
\(\lambda(\wrd)=(r_{1},\ldots,r_{|\wrd|_{\tta}})\), where \(r_{i}\) is the number of
\(\ttb\)'s to the right of the \(i^{\text{th}}\) \(\tta\) in \(\wrd\), with
\(|\lambda(\wrd)|=\binom{\wrd}{\tta\ttb}\).  It is a partition in the
\(|\wrd|_{\tta}\times |\wrd|_{\ttb}\) rectangle; its height--area pair
\(\bigl(|\wrd|_{\tta},|\lambda(\wrd)|\bigr)\) is the complete
\(\sim_{2}\)-invariant of \cref{prop:complete-invariant}.

Conversely, given \(0\le j\le n\) and a partition \(\lambda=(r_{1},\ldots,r_{j})\)
in the \(j\times (n-j)\) rectangle, define
\[
    W(\lambda):=\ttb^{\,(n-j)-r_{1}}\tta\,\ttb^{\,r_{1}-r_{2}}\tta\cdots
    \ttb^{\,r_{j-1}-r_{j}}\tta\,\ttb^{\,r_{j}},
\]
with \(W(\varnothing):=\ttb^{\,n}\).
\begin{example}
For \(\wrd=\ttb\ttb\tta\ttb\ttb\tta\ttb\tta\),
\(\lambda(\wrd)=(3,1,0)\) and \(|\lambda(\wrd)|=4\).
Conversely,
\[
    W(3,1,0)=\ttb^{\,5-3}\tta\,\ttb^{\,3-1}\tta\,\ttb^{\,1-0}\tta\,\ttb^{\,0}=\wrd.
\]
\end{example}
Since \(\lambda(W(\lambda))=\lambda\) and \(W(\lambda(\wrd))=\wrd\), the map
\(\lambda\) is a bijection with inverse \(W\).

Recall from \cref{def:s-tail} that \(\ell=\min(s,|\wrd|_{\tta})\), and that
\(\theta\) and \(|\mathcal{T}_s(\wrd)|\) denote the last entry and the sum of the
\(s\)-tail \(\mathcal{T}_s(\wrd)\), respectively.

\begin{definition}
\label{def:s-residual-area}
Let \(\wrd\in\{\tta,\ttb\}^n\).  The \(s\)-\emph{residual area} of
\(\wrd\) is defined as
\[
R_s(\wrd):=\binom{\wrd}{\tta\ttb}-(|\wrd|_{\tta}-\ell)\theta-|\mathcal{T}_s(\wrd)|.
\]
\end{definition}

\begin{lemma}\label{lem:residual-area-bounds}
Let \(\wrd\in\{\tta,\ttb\}^n\).  Then we have
\[
    0\le R_s(\wrd)\le (|\wrd|_\tta-\ell)(|\wrd|_\ttb-\theta).
\]
\end{lemma}

\begin{proof}
The formula gives \(R_s(\wrd)\) as the number of cells of \(\lambda(\wrd)\) above height
\(\theta\) in the \(|\wrd|_\tta-\ell\) unrecorded rows, which fit inside a
\((|\wrd|_\tta-\ell)\times(|\wrd|_\ttb-\theta)\) rectangle; see \cref{fig:residual-area}.
\end{proof}

\begin{example}
\label{ex:residual-area-bound}
Let \(\wrd=\tta\ttb\tta\ttb\ttb\tta\ttb\tta\ttb\) such that \(s=2\), one has
\(\lambda(\wrd)=(5,4,2,1)\) and \(\ell=\theta=2\); \(R_2(\wrd)=5\)
fits inside the \((|\wrd|_\tta-\ell)\times(|\wrd|_\ttb-\theta)=2\times 3\) rectangle of
\cref{lem:residual-area-bounds}; see \cref{fig:residual-area}.

\begin{center}
\definecolor{rdInk}{HTML}{3E2E38}    
\definecolor{rdGold}{HTML}{E36A48}   
\definecolor{rdGoldF}{HTML}{F6BC8C}  
\definecolor{rdGrey}{HTML}{F4EDE1}   
\definecolor{rdBlue}{HTML}{8E4A78}   
\definecolor{rdBlueF}{HTML}{CBA3C2}  
\definecolor{rdFrame}{HTML}{CFC4CC}  
\begin{tikzpicture}[x=0.74cm, y=0.74cm,
    ce/.style={draw=rdInk, line width=0.5pt},
    lab/.style={font=\small, text=rdInk},
    sub/.style={font=\footnotesize, text=rdInk}]
\fill[rdGrey]  (0,-1) rectangle (2,1);
\fill[rdGoldF] (2,0)  rectangle (5,1);
\fill[rdGoldF] (2,-1) rectangle (4,0);
\fill[rdBlueF] (0,-2) rectangle (2,-1);
\fill[rdBlueF] (0,-3) rectangle (1,-2);
\draw[draw=rdFrame, line width=0.5pt, dash pattern=on 1.5pt off 1.8pt]
  (0,1) rectangle (5,-3);
\foreach \j in {1,...,5} \draw[ce] (\j-1,0)  rectangle (\j,1);
\foreach \j in {1,...,4} \draw[ce] (\j-1,-1) rectangle (\j,0);
\foreach \j in {1,2}     \draw[ce] (\j-1,-2) rectangle (\j,-1);
\draw[ce] (0,-3) rectangle (1,-2);
\draw[draw=rdGold, line width=1.1pt, dash pattern=on 3pt off 2pt]
  (2,1) rectangle (5,-1);
\draw[decorate, decoration={brace, amplitude=4pt}, rdInk, line width=0.6pt]
  (0,1.24) -- (2,1.24) node[midway, above=3pt, lab] {$\theta$};
\draw[decorate, decoration={brace, amplitude=4pt, mirror}, rdInk, line width=0.6pt]
  (-0.26,1) -- (-0.26,-1)
  node[midway, left=4pt, sub, align=right] {unrecorded\\rows};
\draw[decorate, decoration={brace, amplitude=4pt, mirror}, rdBlue, line width=0.6pt]
  (-0.26,-1) -- (-0.26,-3)
  node[midway, left=4pt, font=\footnotesize, text=rdBlue, align=right]
  {$s$-tail\\$\mathcal{T}_s(\wrd)$};
\node[font=\small\bfseries, text=rdGold, anchor=west] at (5.4,0.45) {$R_s(\wrd)$};
\draw[rdGold, line width=0.7pt, -{Latex[length=1.8mm]}]
  (5.35,0.47) to[out=180, in=22] (4.05,0.55);
\end{tikzpicture}

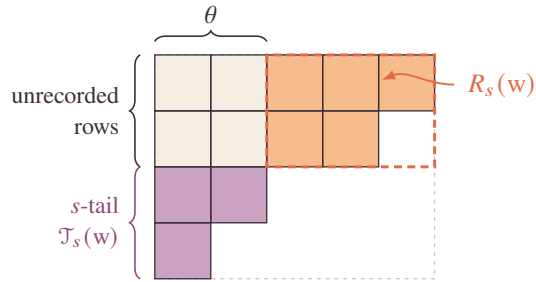
\captionof{figure}{Decomposition of \(\lambda(\wrd)\) into the \(s\)-tail, the
forced \(\theta\)-columns, and the residual area \(R_s(\wrd)\).}
\label{fig:residual-area}
\end{center}
\end{example}

Recall from \cref{def:approx-s} that \(\wrd\approx_s\wrd'\) means
\(\wrd\sim_2\wrd'\) together with \(\mathcal{T}_s(\wrd)=\mathcal{T}_s(\wrd')\),
the class being written \([\wrd]_s\).

\begin{proposition}
\label{prop:equiv-s-chain}
The \(\approx_s\) form a decreasing chain \(\approx_{s+1}\subseteq\approx_s\),
with \(\approx_0=\sim_2\) and \(\approx_n\) the equality relation.
\end{proposition}

\begin{proof}
Since \(\mathcal{T}_s(\wrd)\) is a prefix of \(\mathcal{T}_{s+1}(\wrd)\) and the
\(\sim_2\) condition is independent of \(s\), we have
\(\approx_{s+1}\subseteq\approx_s\).  For \(s=0\), \(\mathcal{T}_0(\wrd)=\varnothing\),
so \(\approx_0=\sim_2\).  For \(s=n\), \(\mathcal{T}_n(\wrd)=\lambda(\wrd)\); since
\(n\) and \(\lambda(\wrd)\) determine \(\wrd\) (\cref{subsec:young}), \(\approx_n\)
is equality.
\end{proof}

\subsection{The chamber-to-class map}
\label{subsec:explicit-chamber-to-class-map}

Let \(d\geq 3\) and put \(s=d-3\).  For a chamber
\(C\in\operatorname{Ch}(\mathcal{H}_n^{(d)})\) with sign word
\(\sigma(C)=\sigma_1\cdots\sigma_n\in\{+,-\}^n\), the \emph{address} of \(C\) is
the set of sign-change positions of the augmented word \(+\sigma(C)\):
\[
    \Delta(C):=\{\,1:\sigma_1=-\,\}\cup\{\,i\in\intv{2}{n}:\sigma_{i-1}\neq\sigma_i\,\}.
\]

\medskip\noindent\textbf{Notation.}\enspace
For a set \(X\), we write \(\binom{X}{k}:=\{\Delta\subseteq X:\#\Delta=k\}\) and,
following Gerbner and Patk\'{o}s~\cite{GerbnerPatkos2018},
\(\binom{X}{\leq r}:=\bigcup_{k=0}^{r}\binom{X}{k}\).

Then \(\#\Delta(C)+1\) is the number of runs of \(+\sigma(C)\), so
\(\Delta(C)\in\binom{\intv{1}{n}}{\leq d}\) by \cref{thm:run-bijection}; conversely,
every \(\Delta\subseteq\intv{1}{n}\) is the address of a unique sign word, since the
augmented word begins with \(+\).

Given \(\Delta=\{c_1<\cdots<c_r\}\subseteq\intv{1}{n}\) with \(r\leq s+3\), its
smallest \(\min(r,s)\) elements give the \(s\)-tail entries \(c_i-i\) and the rest
encode the residual area.  We define \(\lambda_\Delta\) as follows.

\begin{itemize}
\item \textbf{Case~I: \(r\leq s\)} (all of \(\Delta\) is tail).  Then
\[
    \lambda_\Delta:=(c_r-r,\,c_{r-1}-(r-1),\,\ldots,\,c_1-1).
\]

\item \textbf{Case~II: \(r=s+1\)} (zero residual area).  For \(s\in\N_{>0}\), the
largest tail entry \(c_s-s\) is repeated above the remaining tail entries:
\[
    \lambda_\Delta:=\bigl((c_s-s)^{\,n-c_{s+1}+2},\,c_{s-1}-(s-1),\,\ldots,\,c_1-1\bigr),
\]
while for \(s=0\),
\[
    \lambda_\Delta:=(0^{\,n-c_1+1}).
\]

\item \textbf{Case~III: \(r\in\{s+2,s+3\}\)} (positive residual area).  Adjoin
\(n+1\) to \(\Delta\) when \(r=s+2\), and write the resulting
set, which is \(\Delta\) itself when \(r=s+3\), as
\(\widetilde{\Delta}=\{c_1<\cdots<c_{s+3}\}\).  For \(s\in\N_{>0}\),
\[
    \lambda_\Delta:=\bigl((c_{s+2}-s-1)^{\,c_{s+3}-c_{s+2}-1},\,c_{s+1}-s,\,(c_s-s)^{\,n-c_{s+3}+2},\,c_{s-1}-(s-1),\,\ldots,\,c_1-1\bigr),
\]
while for \(s=0\),
\[
    \lambda_\Delta:=\bigl((c_2-1)^{\,c_3-c_2-1},\,c_1,\,0^{\,n-c_3+1}\bigr).
\]
The ordering \(c_1<\cdots<c_{s+3}\leq n+1\) ensures that
\(\lambda_\Delta\) is a partition.
\end{itemize}

\begin{example}
Let \(s=1\) and \(n=5\), and let \(C\) have sign word
\(\sigma(C)=\mathord{+}\mathord{-}\mathord{-}\mathord{+}\mathord{-}\); its
augmented word changes sign at positions \(2,4,5\), so
\(\Delta(C)=\{2,4,5\}\).  Here \(r=3=s+2\), so Case~III applies with
\(\widetilde{\Delta}=\{2,4,5,6\}\), giving
\(\lambda_\Delta=(3^{0},\,3,\,1^{1})=(3,1)\) and
\(W(\lambda_\Delta)=\tta\ttb\ttb\tta\ttb\).
\end{example}

The construction of Cases~I--III sends an address \(\Delta\) to a partition
\(\lambda_\Delta\); composing it with the partition--word bijection \(W\) of
\cref{subsec:young} and taking the \(\approx_s\)-class gives the map
\[
    \begin{aligned}
        \Lambda_{s,n} \colon \binom{\intv{1}{n}}{\leq s+3}
            &\longrightarrow \{\tta,\ttb\}^{n} / {\approx_s}, \\
        \Delta &\longmapsto \bigl[\,W(\lambda_\Delta)\,\bigr]_s.
    \end{aligned}
\]
In words, \(\Lambda_{s,n}\) reads the partition \(\lambda_\Delta\) off \(\Delta\),
converts it into the word \(W(\lambda_\Delta)\), and returns that word's
\(\approx_s\)-class.  The cardinality bound \(\#\Delta\leq s+3\) is exactly the
run bound of \cref{thm:run-bijection} at \(d=s+3\), so the domain of
\(\Lambda_{s,n}\) is the set of chamber addresses.

\begin{lemma}\label{lem:Lambda-bijective}
The map \(\Lambda_{s,n}\) is a bijection.
\end{lemma}

\begin{proof}
We construct the inverse \([\wrd]_s\mapsto\Delta\).  Let \(\wrd\) represent
\([\wrd]_s\) and write \(\lambda(\wrd)=(r_1,\ldots,r_{|\wrd|_\tta})\).
\begin{itemize}
\item \textit{Case 1 (\(|\wrd|_\tta\leq s\)).}
The set
\[
\Delta:=\{r_{|\wrd|_\tta}+1,\,r_{|\wrd|_\tta-1}+2,\,\ldots,\,r_1+|\wrd|_\tta\}
\]
has strictly increasing elements with maximum \(r_1+|\wrd|_\tta\leq n\), so
\(\Delta\in\binom{\intv{1}{n}}{|\wrd|_\tta}\) and \(\lambda_\Delta=\lambda(\wrd)\).

\item \textit{Case 2 (\(|\wrd|_\tta>s\), \(R_s(\wrd)=0\)).}
Here \(r_1=\cdots=r_{|\wrd|_\tta-s}=\theta\), and \(\lambda_\Delta=\lambda(\wrd)\) for
\[
\Delta:=\{r_{|\wrd|_\tta}+1,\,r_{|\wrd|_\tta-1}+2,\,\ldots,\,
r_{|\wrd|_\tta-s+1}+s,\,|\wrd|_\ttb+s+1\}\in\binom{\intv{1}{n}}{s+1}.
\]

\item \textit{Case 3 (\(|\wrd|_\tta>s\), \(R_s(\wrd)>0\)).}
Then \(|\wrd|_\ttb>\theta\), and \cref{lem:residual-area-bounds} gives
\(R_s(\wrd)\leq(|\wrd|_\tta-s)(|\wrd|_\ttb-\theta)\), so \(R_s(\wrd)\) has the unique
expression
\[
R_s(\wrd)=(u-1)(|\wrd|_\ttb-\theta)+(x-\theta),
\qquad 1\leq u\leq |\wrd|_\tta-s,\qquad \theta<x\leq |\wrd|_\ttb.
\]
With
\[
\widetilde{\Delta}:=\{r_{|\wrd|_\tta}+1,\,r_{|\wrd|_\tta-1}+2,\,\ldots,\,
r_{|\wrd|_\tta-s+1}+s,\,x+s,\,|\wrd|_\ttb+s+1,\,|\wrd|_\ttb+u+s+1\},
\]
set
\[
\Delta:=
\begin{cases}
\widetilde{\Delta}, & \text{if } |\wrd|_\ttb+u+s+1\leq n,\\
\widetilde{\Delta}\setminus\{n+1\}, & \text{if } |\wrd|_\ttb+u+s+1=n+1,
\end{cases}
\]
so \(\#\Delta\in\{s+2,s+3\}\).  Then \(\lambda_\Delta\) shares the \(|\wrd|_\tta\),
\(s\)-tail, and residual area of \(\lambda(\wrd)\), so \(W(\lambda_\Delta)\approx_s\wrd\).
\end{itemize}
In each case \(\Delta\) is determined by \(|\wrd|_\tta\), the \(s\)-tail, and
\(R_s(\wrd)\), hence by \([\wrd]_s\) alone, and satisfies
\(\Lambda_{s,n}(\Delta)=[\wrd]_s\); thus \([\wrd]_s\mapsto\Delta\) is the inverse of
\(\Lambda_{s,n}\).
\end{proof}

\subsection{The chamber-to-class theorem for \texorpdfstring{\(d\ge 3\)}{d >= 3}}
\label{subsec:main-theorem-d-geq-3}

Let \(d\geq3\) and put \(s=d-3\).  Composing the address bijection
\(C\mapsto\Delta(C)\) of \cref{thm:run-bijection} with \(\Lambda_{s,n}\) yields the
\emph{chamber-to-class map}
\[
    \begin{aligned}
        \Psi_{d,n} \colon \operatorname{Ch}\bigl(\mathcal{H}_n^{(d)}\bigr)
            &\longrightarrow \{\tta,\ttb\}^{n} / {\approx_s}, \\
        C &\longmapsto [w_C]_s,
    \end{aligned}
\]
where \(w_C:=W(\lambda_{\Delta(C)})\) is the \emph{canonical representative} of
\(C\).

\begin{theorem}
\label{thm:higher-tail-schlaefli-dictionary}
The map \(\Psi_{d,n}\) is a bijection.
\end{theorem}

\begin{proof}
By construction \(\Psi_{d,n}=\Lambda_{s,n}\circ\Delta\).  The address map \(\Delta\)
is a bijection by \cref{thm:run-bijection} and \(\Lambda_{s,n}\) by
\cref{lem:Lambda-bijective} (as \(d=s+3\)), hence so is \(\Psi_{d,n}\).
\end{proof}

Thus the words \(w_C\), for \(C\in\operatorname{Ch}(\mathcal{H}_n^{(d)})\),
form a complete set of representatives of \(\{\tta,\ttb\}^n/\!\approx_s\).  Taking
cardinalities gives, for every \(s\in\N\),
\[
\#\bigl(\{\tta,\ttb\}^n/\approx_s\bigr)=\sum_{r=0}^{s+3}\binom{n}{r}.
\]

\begin{corollary}
\label{cor:equiv-s-stabilisation}
Let \(n,s\in\N\).  Then \(\approx_{s+1}\subsetneq\approx_s\) iff \(s\leq n-4\);
equivalently, \(\approx_s\) is equality on \(\{\tta,\ttb\}^n\) iff \(s\geq n-3\).
\end{corollary}

\begin{proof}
By the count above, the numbers of \(\approx_{s+1}\)- and \(\approx_s\)-classes
differ by \(\binom{n}{s+4}\); since \(\approx_{s+1}\subseteq\approx_s\) by
\cref{prop:equiv-s-chain}, the inclusion is strict iff \(\binom{n}{s+4}>0\), that
is \(s\leq n-4\).  For \(s\geq n-3\) the count is \(2^{n}=\#\{\tta,\ttb\}^n\),
so every class is a singleton and \(\approx_s\) is equality.
\end{proof}

\subsection{The cases \texorpdfstring{\(d=3\)}{d = 3} and
\texorpdfstring{\(d=4\)}{d = 4}}
\label{subsec:cases-d3-d4}\label{subsec:cake-d3}

The cases \(d=3\) and \(d=4\) of
\cref{thm:higher-tail-schlaefli-dictionary} recover, respectively,
\(2\)-binomial equivalence (the cake numbers) and \(1\)-quasi-binomial
equivalence.

\medskip\noindent\textbf{\(\mathbf{2}\)-binomial equivalence: \(\mathbf{d=3}\).}\enspace
Here \(s=0\), so \(\approx_0=\sim_2\) by \cref{prop:equiv-s-chain}, and \(\Psi_{3,n}\)
identifies the chambers of \(\mathcal{H}_n^{(3)}\) with the \(\sim_2\)-classes of
\(\{\tta,\ttb\}^n\), counted by the cake numbers
\(1,2,4,8,15,26,\ldots\) (\seqnum{A000125}).
At \(s=0\), the construction of \cref{subsec:explicit-chamber-to-class-map} gives
\[
w_C=
\begin{cases}
\ttb^{n}, & \text{if } \Delta(C)=\varnothing,\\
\ttb^{c-1}\tta^{n-c+1}, & \text{if } \Delta(C)=\{c\},\\
\tta^{c_3-c_2-1}\ttb^{c_2-c_1-1}\tta\ttb^{c_1}\tta^{n-c_3+1}, & \text{if } \#\Delta(C)\in\{2,3\},
\end{cases}
\]
where \(\widetilde{\Delta}=\{c_1<c_2<c_3\}\) is \(\Delta(C)\cup\{n+1\}\) if
\(\#\Delta(C)=2\) and \(\Delta(C)\) if \(\#\Delta(C)=3\).

\begin{example}
Take \(n=6\) and let \(C\) have sign word
\(\sigma(C)=\mathord{+}\mathord{-}\mathord{-}\mathord{+}\mathord{+}\mathord{+}\),
so \(\Delta(C)=\{2,4\}\).  Since \(\#\Delta(C)=2\) we adjoin \(n+1=7\), giving
\(\widetilde{\Delta}=\{2,4,7\}\); hence
\(W(\lambda_{\Delta(C)})=\tta\tta\ttb\tta\ttb\ttb\).
\end{example}

\cref{fig:cake-d3} illustrates this bijection for \(n=3\): \(\Psi_{3,3}\) matches
the eight chambers of \(\mathcal{H}_3^{(3)}\) with the \(\sim_2\)-classes of
\(\{\tta,\ttb\}^3\).

\begin{figure}[!htbp]
\centering
\definecolor{ckbbb}{HTML}{B83563}   
\definecolor{ckbba}{HTML}{2E3F8F}   
\definecolor{ckbab}{HTML}{1F7A7A}   
\definecolor{ckabb}{HTML}{2E7D3A}   
\definecolor{ckbaa}{HTML}{2E6CB8}   
\definecolor{ckaba}{HTML}{9C6B1F}   
\definecolor{ckaab}{HTML}{B83560}   
\definecolor{ckaaa}{HTML}{7E8C2B}   
\definecolor{ckRule}{HTML}{3E2E38}  
\begin{minipage}[c]{0.25\textwidth}
\centering
\small
\renewcommand{\arraystretch}{1.45}
\setlength{\tabcolsep}{3pt}
\begin{tabular}{@{}clccc@{}}
 & \textbf{word} & {\boldmath$\sigma$} & {\boldmath$|\wrd|_{\tta}$} & {\boldmath$\binom{\wrd}{\tta\ttb}$}\\
\noalign{\global\setlength{\arrayrulewidth}{0.9pt}}\cmidrule(lr){2-5}%
\noalign{\global\setlength{\arrayrulewidth}{0.4pt}}
\textcolor{ckbbb}{$\bullet$} & \textcolor{ckbbb}{$\ttb\ttb\ttb$}
  & \textcolor{ckbbb}{$+\,+\,+$} & \textcolor{ckbbb}{$0$} & \textcolor{ckbbb}{$0$}\\
\textcolor{ckbba}{$\bullet$} & \textcolor{ckbba}{$\ttb\ttb\tta$}
  & \textcolor{ckbba}{$-\,+\,+$} & \textcolor{ckbba}{$1$} & \textcolor{ckbba}{$0$}\\
\textcolor{ckbab}{$\bullet$} & \textcolor{ckbab}{$\ttb\tta\ttb$}
  & \textcolor{ckbab}{$+\,-\,+$} & \textcolor{ckbab}{$1$} & \textcolor{ckbab}{$1$}\\
\textcolor{ckabb}{$\bullet$} & \textcolor{ckabb}{$\tta\ttb\ttb$}
  & \textcolor{ckabb}{$+\,+\,-$} & \textcolor{ckabb}{$1$} & \textcolor{ckabb}{$2$}\\
\textcolor{ckbaa}{$\bullet$} & \textcolor{ckbaa}{$\ttb\tta\tta$}
  & \textcolor{ckbaa}{$-\,-\,+$} & \textcolor{ckbaa}{$2$} & \textcolor{ckbaa}{$0$}\\
\textcolor{ckaba}{$\bullet$} & \textcolor{ckaba}{$\tta\ttb\tta$}
  & \textcolor{ckaba}{$-\,+\,-$} & \textcolor{ckaba}{$2$} & \textcolor{ckaba}{$1$}\\
\textcolor{ckaab}{$\bullet$} & \textcolor{ckaab}{$\tta\tta\ttb$}
  & \textcolor{ckaab}{$+\,-\,-$} & \textcolor{ckaab}{$2$} & \textcolor{ckaab}{$2$}\\
\textcolor{ckaaa}{$\bullet$} & \textcolor{ckaaa}{$\tta\tta\tta$}
  & \textcolor{ckaaa}{$-\,-\,-$} & \textcolor{ckaaa}{$3$} & \textcolor{ckaaa}{$0$}\\
\end{tabular}
\end{minipage}\hfill
\begin{minipage}[c]{0.74\textwidth}
\centering
\includegraphics[width=\linewidth]{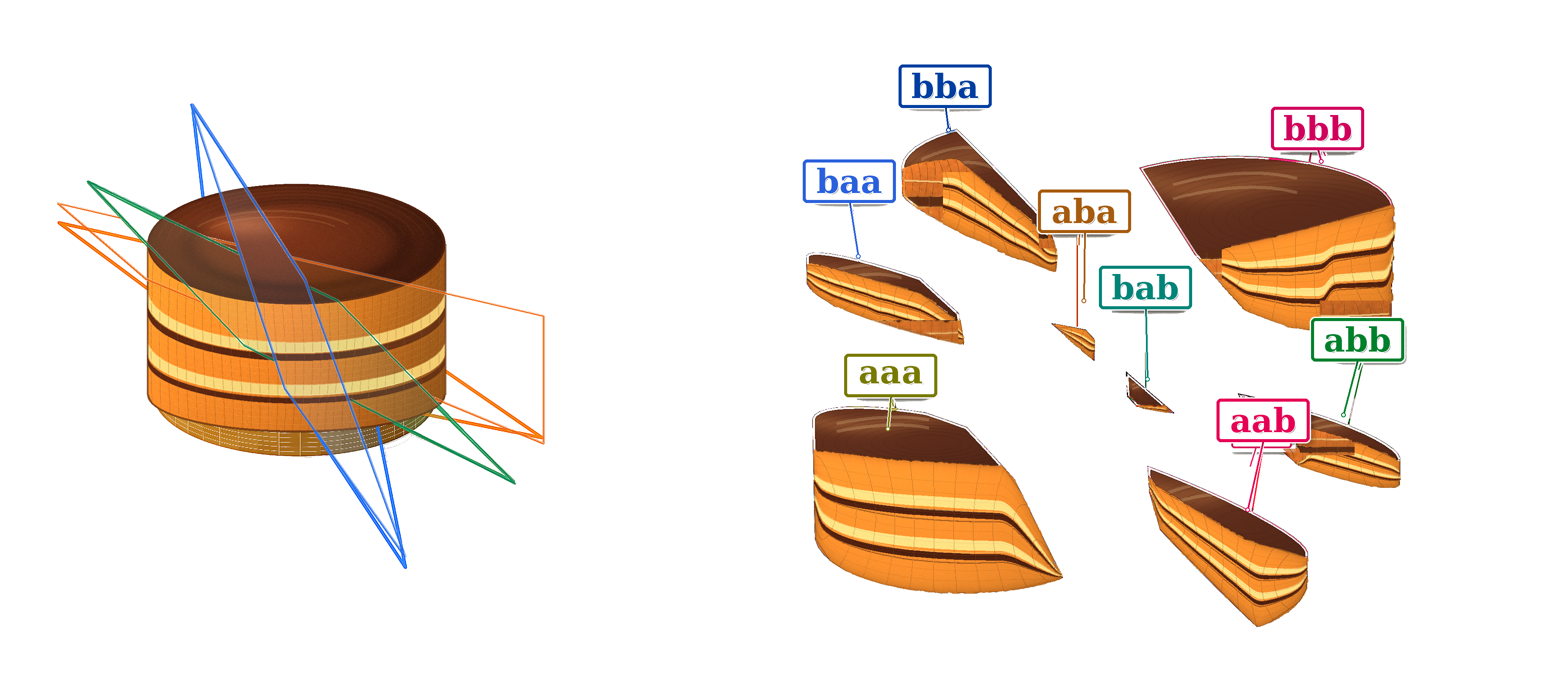}

\medskip
\captionof{figure}{Cake dissection for \(n=3\).}
\label{fig:cake-d3}
\end{minipage}
\end{figure}

\medskip\noindent\textbf{\(\mathbf{1}\)-tail refinement: \(\mathbf{d=4}\).}\enspace
Here \(s=1\), so \(\approx_1\) refines \(\sim_2\) by recording the \(1\)-tail
\(\mathcal{T}_1(\wrd)=(\theta)\), where \(\theta\) is the number of \(\ttb\)'s
after the final \(\tta\) (and \(\mathcal{T}_1(\ttb^n)=\varnothing\)).  Thus
\(\Psi_{4,n}\) identifies the chambers of \(\mathcal{H}_n^{(4)}\) with the
\(\approx_1\)-classes of \(\{\tta,\ttb\}^n\), counted by the circle-cutting
numbers \(1,2,4,8,16,31,\ldots\) (\seqnum{A000127}).
At \(s=1\), the construction of \cref{subsec:explicit-chamber-to-class-map} gives
\[
W(\lambda_{\Delta(C)})=
\begin{cases}
\ttb^{n}, & \text{if } \Delta(C)=\varnothing,\\
\ttb^{n-c_1}\tta\ttb^{c_1-1}, & \text{if } \Delta(C)=\{c_1\},\\
\ttb^{c_2-c_1-1}\tta^{n-c_2+2}\ttb^{c_1-1}, & \text{if } \Delta(C)=\{c_1<c_2\},\\
\tta^{c_4-c_3-1}\ttb^{c_3-c_2-1}\tta\ttb^{c_2-c_1}\tta^{n-c_4+2}\ttb^{c_1-1}, & \text{if } \#\Delta(C)\in\{3,4\},
\end{cases}
\]
where \(\widetilde{\Delta}=\{c_1<c_2<c_3<c_4\}\) is \(\Delta(C)\cup\{n+1\}\) if
\(\#\Delta(C)=3\) and \(\Delta(C)\) if \(\#\Delta(C)=4\).

\begin{example}
Take \(n=7\) and let \(C\) have sign word
\(\sigma(C)=\mathord{+}\mathord{-}\mathord{-}\mathord{+}\mathord{+}\mathord{-}\mathord{-}\),
so \(\Delta(C)=\{2,4,6\}\).  Since \(\#\Delta(C)=3\) we adjoin \(n+1=8\), giving
\(\widetilde{\Delta}=\{2,4,6,8\}\); hence
\(W(\lambda_{\Delta(C)})=\tta\ttb\tta\ttb\ttb\tta\ttb\).
\end{example}

\begin{remark}
For \(d=4\), the count \(\sum_{r=0}^{4}\binom{n}{r}\) carries two combinatorial
meanings.  Here it is the number of chambers of \(\mathcal{H}_n^{(4)}\), matched
above with the \(\approx_1\)-classes of \(\{\tta,\ttb\}^n\).  By Pascal's identity,
\[
\sum_{r=0}^{4}\binom{n}{r}=1+\binom{n+1}{2}+\binom{n+1}{4},
\]
the same number is the circle-cutting number~(\seqnum{A000127}): the regions
formed by the chords of \(n+1\) points in general position on a circle.  A
bijection with this dissection can likewise be given, but we omit it.
\end{remark}

\section{The cases \texorpdfstring{\(d=1\)}{d = 1} and
\texorpdfstring{\(d=2\)}{d = 2}}\label{sec:lower-dim}

For each \(d\geq 3\), \cref{sec:higher-tail} matched the chambers of
\(\mathcal{H}_{n}^{(d)}\) with the classes of an equivalence relation on
\(\{\tta,\ttb\}^n\).  This section treats \(d=1\) and \(d=2\): the first is
abelian equivalence \(\sim_1\), the second a new relation \(\sim_{3/2}\).

\subsection{The case \texorpdfstring{\(d=1\)}{d = 1}: abelian equivalence}\label{subsec:abelian-d1}

By bars and stars~\cite[Section~1.2]{Sta12}, \(\{\tta,\ttb\}^n\) has \(n+1\)
abelian classes~(\seqnum{A000027}), matching the chamber count of
\(\mathcal{H}_{n}^{(1)}\) by \cref{thm:SS}.

By \cref{thm:run-bijection}, \(+\sigma(C)\) has at most two runs, so each
chamber sign word has the form \(\sigma(C)={+}^{\,n-r(C)}\,{-}^{\,r(C)}\) such that
\(r(C)\in\intv{0}{n}\).  Set \(w_C:=\ttb^{\,n-r(C)}\,\tta^{\,r(C)}\) and define
\[
    \Psi_{1,n} \colon \operatorname{Ch}\bigl(\mathcal{H}_{n}^{(1)}\bigr)
    \longrightarrow \{\tta,\ttb\}^{n} / {\sim_{1}},
    \qquad
    C \longmapsto [w_C]_{\sim_{1}}.
\]

\cref{fig:d1_arrangement} illustrates this bijection for \(n=3\).

\begin{figure}[!ht]
\centering
\scalebox{0.7}{%
\begin{tikzpicture}[scale=1.0,
    pointlabel/.style={font=\scriptsize, text=daInk},
    signword/.style={font=\small, text=daInk},
    jvalue/.style={font=\scriptsize, text=daInk}]
\definecolor{daInk}{HTML}{3E2E38}   
\definecolor{daC3}{HTML}{E36A48}    
\definecolor{daC2}{HTML}{EE9468}    
\definecolor{daC1}{HTML}{F6BC8C}    
\definecolor{daC0}{HTML}{FBEAD8}    
\def\xa{0}\def\xb{2.6}\def\xc{5.2}\def\xd{7.8}\def\xe{10.4}
\def\xfL{-1.2}\def\xfR{11.6}  
\def\hh{0.42}                 
\shade[left color=white, right color=daC3] (\xfL,-\hh) rectangle (\xa,\hh);
\fill[daC3] (\xa,-\hh) rectangle (\xb,\hh);
\fill[daC2] (\xb,-\hh) rectangle (\xc,\hh);
\fill[daC1] (\xc,-\hh) rectangle (\xd,\hh);
\fill[daC0] (\xd,-\hh) rectangle (\xe,\hh);
\shade[left color=daC0, right color=white] (\xe,-\hh) rectangle (\xfR,\hh);
\foreach \x in {\xb,\xc,\xd}{%
  \draw[daInk, line width=0.5pt, opacity=0.5] (\x,-\hh-0.14) -- (\x,\hh+0.14);}
\draw[daInk, line width=1pt] (\xfL-0.5,0) -- (\xfR+0.5,0);
\draw[daInk, line width=1pt, ->] (\xfR+0.5,0) -- (\xfR+0.95,0);
\draw[daInk, line width=1pt, <-] (\xfL-0.95,0) -- (\xfL-0.5,0);
\foreach \x in {\xb,\xc,\xd}{%
  \fill[white] (\x,0) circle (4pt);
  \fill[daInk] (\x,0) circle (2.6pt);}
\node[signword, above=16pt] at ({(\xfL+\xb)/2},0) {$-\,-\,-$};
\node[signword, above=16pt] at ({(\xb+\xc)/2},0) {$+\,-\,-$};
\node[signword, above=16pt] at ({(\xc+\xd)/2},0) {$+\,+\,-$};
\node[signword, above=16pt] at ({(\xd+\xfR)/2},0) {$+\,+\,+$};
\node[pointlabel, below=14pt] at (\xb,0) {$-1$};
\node[pointlabel, below=14pt] at (\xc,0) {$-\tfrac{1}{2}$};
\node[pointlabel, below=14pt] at (\xd,0) {$-\tfrac{1}{3}$};
\node[jvalue, below=30pt] at ({(\xfL+\xb)/2},0) {$|w_C|_\tta=3$};
\node[jvalue, below=30pt] at ({(\xb+\xc)/2},0) {$|w_C|_\tta=2$};
\node[jvalue, below=30pt] at ({(\xc+\xd)/2},0) {$|w_C|_\tta=1$};
\node[jvalue, below=30pt] at ({(\xd+\xfR)/2},0) {$|w_C|_\tta=0$};
\end{tikzpicture}}
\caption{The arrangement $\mathcal{H}_3^{(1)}$, with chamber sign words and
letter counts $|w_C|_\tta$.}
\label{fig:d1_arrangement}
\end{figure}
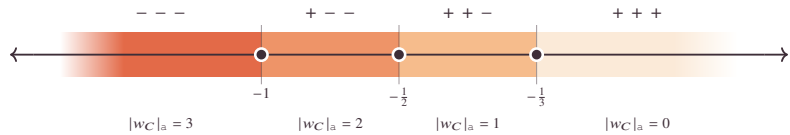

\begin{proposition}\label{thm:d1bijection}
The map $\Psi_{1,n}$ is a bijection.
\end{proposition}

\begin{proof}
We have \(|w_C|_\tta = r(C)\).  Since an abelian class in
\(\{\tta,\ttb\}^n\) is determined by \(|\wrd|_\tta\in\intv{0}{n}\), and since
\(r(C)\) takes each value in \(\intv{0}{n}\) exactly once, the map is bijective.
\end{proof}

Hence $\#\bigl(\{\tta,\ttb\}^n/{\sim_1}\bigr)=n+1$, the chamber count of
$\mathcal{H}_n^{(1)}$.

\begin{remark}
In the Young-diagram picture, $\wu\sim_1\wrd$ iff $\lambda(\wu)$ and
$\lambda(\wrd)$ have the same number of rows.
\end{remark}

\subsection{The case \texorpdfstring{\(d=2\)}{d = 2}: \texorpdfstring{\(3/2\)}{3/2}-binomial equivalence}
\label{subsec:three-halves-d2}\label{subsec:def-32}

Suppose that \(\wu,\wrd\in\{\tta,\ttb\}^n\).  Then \cref{def:sim-32} specialises to
\[
   \wu\sim_{3/2}\wrd
   \quad\Longleftrightarrow\quad
   |\wu|_{\tta}=|\wrd|_{\tta}
   \text{ and }
   \operatorname{strip}(\wu)=\operatorname{strip}(\wrd).
\]
Thus \(\sim_{3/2}\) coarsens \(\sim_2\) by replacing \(\binom{\wrd}{\tta\ttb}\)
with \(\operatorname{strip}(\wrd)\).

An equivalence \(\equiv\) on \(\{\tta,\ttb\}^{*}\)
is a \emph{concatenation congruence} if \(\wu\equiv\wu'\) and \(\wv\equiv\wv'\)
imply \(\wu\wv\equiv\wu'\wv'\).

\begin{proposition}
\label{prop:three-halves-chain}
The relation \(\sim_{3/2}\) on \(\{\tta,\ttb\}^{*}\) satisfies
\(\sim_2\subsetneq\sim_{3/2}\subsetneq\sim_1\) and is not a concatenation
congruence.
\end{proposition}

\begin{proof}
The inclusions \(\sim_2\subseteq\sim_{3/2}\subseteq\sim_1\) are immediate from
the definitions, and \cref{ex:three-halves} makes both strict.  For the same
pair \(\wu=\tta\ttb\tta\ttb\), \(\wrd=\tta\ttb\ttb\tta\), adjoining \(\tta\)
yields \(\binom{\wu\tta}{\tta\ttb}=3\), \(\binom{\wrd\tta}{\tta\ttb}=2\), and
\(|\wu\tta|_{\tta}=|\wrd\tta|_{\tta}=3\), so \(\operatorname{strip}(\wu\tta)=1\)
and \(\operatorname{strip}(\wrd\tta)=0\); hence \(\wu\tta\not\sim_{3/2}\wrd\tta\).
\end{proof}

\begin{proposition}
\label{thm:three-halves-count}\label{thm:32-count}
The number of \(\sim_{3/2}\)-classes of \(\{\tta,\ttb\}^n\) is \(1+\binom{n+1}{2}\).
\end{proposition}

\begin{proof}
By \cref{def:sim-32}, a \(\sim_{3/2}\)-class in \(\{\tta,\ttb\}^n\) is determined by the
pair \((|\wrd|_{\tta},\operatorname{strip}(\wrd))\).  The word \(\ttb^{n}\) forms the
unique class with \(|\wrd|_{\tta}=0\).  When \(|\wrd|_{\tta}\geq 1\),
\(\binom{\wrd}{\tta\ttb}\) ranges over \(\intv{0}{|\wrd|_{\tta}\,|\wrd|_{\ttb}}\)
by \cref{prop:complete-invariant}, so \(\operatorname{strip}(\wrd)\) ranges over
\(\intv{0}{|\wrd|_{\ttb}}\) --- that is, \(|\wrd|_{\ttb}+1\) classes for each
nonzero value of \(|\wrd|_{\tta}\).  As \(|\wrd|_{\tta}\) runs through
\(\intv{1}{n}\) these give \(n+(n-1)+\cdots+1=\binom{n+1}{2}\) classes, so
\[
\#\bigl(\{\tta,\ttb\}^n/\sim_{3/2}\bigr)=1+\binom{n+1}{2}.
\]
\end{proof}

\medskip\noindent\textbf{Geometric reformulation.}\enspace
In the Young-diagram picture, \(\lambda(\wrd)\) lies in the \emph{ambient
rectangle}
\[
    \operatorname{Rect}(\wrd) \defeq [0,|\wrd|_{\tta}]\times[0,|\wrd|_{\ttb}],
\]
and has area \(|\lambda(\wrd)|=\binom{\wrd}{\tta\ttb}\).  If
\(|\wrd|_{\tta}\geq 1\), the \emph{mean column-height} of \(\lambda(\wrd)\) is
\[
    \mu(\wrd)\defeq\frac{|\lambda(\wrd)|}{|\wrd|_{\tta}}
                =\frac{\binom{\wrd}{\tta\ttb}}{|\wrd|_{\tta}}
                \in[0,|\wrd|_{\ttb}],
\]
and \(\operatorname{strip}(\wrd)=\lfloor\mu(\wrd)\rfloor\) is the index of the
unit strip of \(\operatorname{Rect}(\wrd)\) met by the mean line
\(y=\mu(\wrd)\), the top value \(\mu(\wrd)=|\wrd|_{\ttb}\) included via the
floor.

The rectangle and strip give a geometric criterion for
\(\sim_{3/2}\)-equivalence.

\begin{proposition}\label{thm:32-tfae}
Let \(\wu,\wrd\in\{\tta,\ttb\}^{n}\).  Then \(\wu\sim_{3/2}\wrd\) iff
\(\operatorname{Rect}(\wu)=\operatorname{Rect}(\wrd)\) and
\(\operatorname{strip}(\wu)=\operatorname{strip}(\wrd)\).
\end{proposition}

\begin{proof}
At fixed length \(n\), \(\operatorname{Rect}(\wu)\) is determined by \(|\wu|_{\tta}\),
so the first condition reads \(|\wu|_{\tta}=|\wrd|_{\tta}\), that is \(\wu\sim_1\wrd\).
Together with \(\operatorname{strip}(\wu)=\operatorname{strip}(\wrd)\), this is exactly
the definition of \(\sim_{3/2}\) in \cref{def:sim-32}, giving \(\wu\sim_{3/2}\wrd\).
\end{proof}

\cref{fig:strip-example} illustrates this on the three words
\(\tta\ttb\tta\ttb\), \(\tta\ttb\ttb\tta\), \(\tta\tta\ttb\ttb\) of
\cref{subsec:def-32}: the first two share \(\operatorname{strip}=1\) and are
therefore \(3/2\)-binomially equivalent, while the third has
\(\operatorname{strip}=2\) and is not.

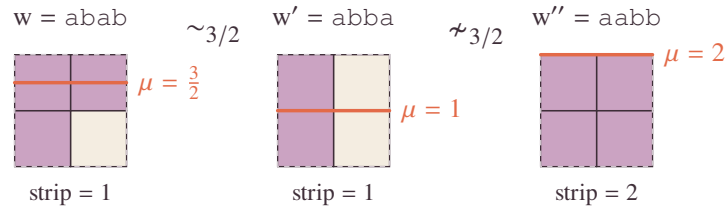
\begin{figure}[ht]
\centering
\begin{tikzpicture}[x=0.74cm, y=0.74cm]
  \definecolor{rdInk}{HTML}{3E2E38}    
  \definecolor{rdGold}{HTML}{E36A48}   
  \definecolor{rdBlueF}{HTML}{CBA3C2}  
  \definecolor{rdGrey}{HTML}{F4EDE1}   
  \definecolor{rdFrame}{HTML}{CFC4CC}  
  \tikzset{
    amb/.style={draw=rdFrame, line width=0.6pt,
                dash pattern=on 2pt off 2pt},
    ce/.style={draw=rdInk, line width=0.5pt},
    meanl/.style={draw=rdGold, line width=1.3pt, line cap=round},
    mlab/.style={font=\small\bfseries, text=rdGold, anchor=west, inner sep=2pt},
    ttl/.style={font=\small, text=rdInk, anchor=south},
    ftr/.style={font=\footnotesize, text=rdInk, anchor=north, inner sep=4pt},
    rel/.style={font=\large, text=rdInk}
  }
  \begin{scope}[shift={(0,0)}]
    \node[ttl] at (1,2.34) {$\wrd=\tta\ttb\tta\ttb$};
    \fill[rdBlueF] (0,1) rectangle (2,2);   
    \fill[rdBlueF] (0,0) rectangle (1,1);   
    \fill[rdGrey]  (1,0) rectangle (2,1);   
    \draw[ce] (0,1) rectangle (1,2);
    \draw[ce] (1,1) rectangle (2,2);
    \draw[ce] (0,0) rectangle (1,1);
    \draw[ce] (1,0) rectangle (2,1);
    \draw[amb] (0,0) rectangle (2,2);
    \draw[meanl] (0,1.5) -- (2,1.5);
    \node[mlab] at (2.1,1.5) {$\mu=\tfrac{3}{2}$};
    \node[ftr] at (1,-0.1) {$\operatorname{strip}=1$};
  \end{scope}
  \node[rel] at (3.55,2.34) {$\sim_{3/2}$};
  \begin{scope}[shift={(4.7,0)}]
    \node[ttl] at (1,2.34) {$\wrd'=\tta\ttb\ttb\tta$};
    \fill[rdBlueF] (0,0) rectangle (1,2);   
    \fill[rdGrey]  (1,0) rectangle (2,2);   
    \draw[ce] (0,0) rectangle (1,1);
    \draw[ce] (0,1) rectangle (1,2);
    \draw[ce] (1,0) rectangle (2,1);
    \draw[ce] (1,1) rectangle (2,2);
    \draw[amb] (0,0) rectangle (2,2);
    \draw[meanl] (0,1) -- (2,1);
    \node[mlab] at (2.1,1) {$\mu=1$};
    \node[ftr] at (1,-0.1) {$\operatorname{strip}=1$};
  \end{scope}
  \node[rel] at (8.25,2.34) {$\not\sim_{3/2}$};
  \begin{scope}[shift={(9.4,0)}]
    \node[ttl] at (1,2.34) {$\wrd''=\tta\tta\ttb\ttb$};
    \fill[rdBlueF] (0,0) rectangle (2,2);   
    \draw[ce] (0,0) rectangle (1,1);
    \draw[ce] (1,0) rectangle (2,1);
    \draw[ce] (0,1) rectangle (1,2);
    \draw[ce] (1,1) rectangle (2,2);
    \draw[amb] (0,0) rectangle (2,2);
    \draw[meanl] (0,2) -- (2,2);
    \node[mlab] at (2.1,2) {$\mu=2$};
    \node[ftr] at (1,-0.1) {$\operatorname{strip}=2$};
  \end{scope}
\end{tikzpicture}
\caption{Geometric criterion for checking $3/2$-equivalence of the words $\wrd$,
$\wrd'$, and $\wrd''$.}
\label{fig:strip-example}
\end{figure}

\medskip\noindent\textbf{Chamber model for \(\boldsymbol{\sim_{3/2}}\).}\enspace
By \cref{thm:three-halves-count}, the \(\sim_{3/2}\)-classes of
\(\{\tta,\ttb\}^n\) are counted by the \emph{lazy caterer's sequence}
\(1,2,4,7,11,16,\ldots\)~(\seqnum{A000124}), matching the chamber count of
\(\mathcal{H}_{n}^{(2)}\) by \cref{thm:SS}.

By \cref{thm:run-bijection}, \(+\sigma(C)\) has at most three runs.  For the
all-plus chamber, set \(r(C)=t(C)=0\); otherwise write uniquely
\(\sigma(C)={+}^{\,t(C)}\,{-}^{\,r(C)}\,{+}^{\,n-r(C)-t(C)}\), where
\(r(C)\ge 1\) is the length of the unique block of minus signs and \(t(C)\) is
the number of leading plus signs.  Set
\[
   w_C:=\ttb^{\,n-r(C)-t(C)}\,\tta^{\,r(C)}\,\ttb^{\,t(C)},
\]
and define
\[
\Psi_{2,n}\colon\operatorname{Ch}\bigl(\mathcal{H}_n^{(2)}\bigr)
       \longrightarrow \{\tta,\ttb\}^n/\sim_{3/2},
       \qquad
       C\longmapsto [w_C]_{\sim_{3/2}}.
\]

\begin{proposition}
\label{thm:d2-three-halves-dictionary}\label{thm:d2bijection}
The map \(\Psi_{2,n}\) is a bijection.
\end{proposition}

\begin{proof}
If \(r(C)=0\), then \(w_C=\ttb^n\) and
\(\bigl(|w_C|_{\tta},\operatorname{strip}(w_C)\bigr)=(0,0)\).  If \(r(C)\ge 1\),
then \(|w_C|_{\tta}=r(C)\) and \(\binom{w_C}{\tta\ttb}=r(C)\,t(C)\), giving
\(\operatorname{strip}(w_C)=\lfloor r(C)\,t(C)/r(C)\rfloor=t(C)\).  Either way,
\(\bigl(|w_C|_{\tta},\operatorname{strip}(w_C)\bigr)=(r(C),t(C))\).  By
\cref{thm:three-halves-count}, the \(\sim_{3/2}\)-classes are indexed by
\[
   \{(0,0)\}\cup
   \{(r,t):r\in\intv{1}{n},\ t\in\intv{0}{n-r}\},
\]
and the chamber parameters \((r(C),t(C))\) run through the same set; hence
\(\Psi_{2,n}\) is bijective.
\end{proof}

\cref{fig:pizza-d2} illustrates this bijection for \(n=3\): \(\Psi_{2,3}\) matches
the seven chambers of \(\mathcal{H}_3^{(2)}\) with the \(\sim_{3/2}\)-classes of
\(\{\tta,\ttb\}^3\).

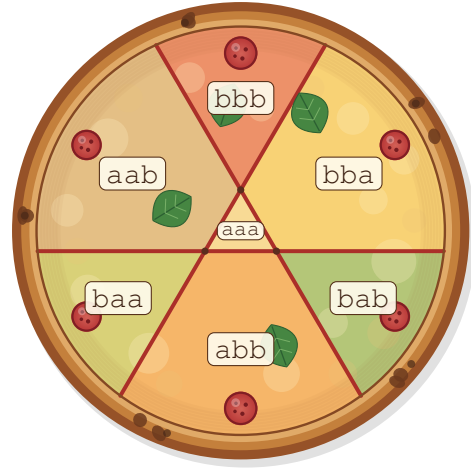
\begin{figure}[ht]
\centering

\definecolor{crustdark}{RGB}{150, 82, 40}
\definecolor{crustmid}{RGB}{196,128, 60}
\definecolor{crusthi}{RGB}{224,170,104}
\definecolor{charA}{RGB}{108, 58, 30}
\definecolor{charB}{RGB}{ 72, 40, 24}
\definecolor{cheese}{RGB}{251,221,148}
\definecolor{cheeselt}{RGB}{255,240,196}
\definecolor{cheesedk}{RGB}{236,196,120}
\definecolor{saucered}{RGB}{171, 47, 41}
\definecolor{pepred}{RGB}{182, 52, 52}
\definecolor{peprim}{RGB}{128, 30, 36}
\definecolor{pephi}{RGB}{210, 96, 86}
\definecolor{pepspeck}{RGB}{120, 28, 32}
\definecolor{basilg}{RGB}{ 70,134, 66}
\definecolor{basildk}{RGB}{ 44, 96, 50}
\definecolor{basilvein}{RGB}{120,176,104}
\definecolor{seedcol}{RGB}{ 92, 44, 30}
\definecolor{wordcol}{RGB}{ 60, 32, 18}
\definecolor{platecol}{RGB}{255,251,236}

\definecolor{cBBB}{RGB}{198, 58, 46}   
\definecolor{cBBA}{RGB}{176,118, 30}   
\definecolor{cBAB}{RGB}{ 60,128, 58}   
\definecolor{cABB}{RGB}{205,110, 40}   
\definecolor{cBAA}{RGB}{138,148, 52}   
\definecolor{cAAB}{RGB}{150,108, 66}   
\definecolor{cAAA}{RGB}{160,120, 45}   

\begin{minipage}[c]{0.43\textwidth}
\centering
\renewcommand{\arraystretch}{1.45}
\setlength{\tabcolsep}{6pt}
\begin{tabular}{@{}clccc@{}}
 & \textbf{word} & {\boldmath$\sigma$} & {\boldmath$|w|_{\tta}$} & {\boldmath$\operatorname{strip}(w)$}\\
\noalign{\global\setlength{\arrayrulewidth}{0.9pt}}\cmidrule(lr){2-5}%
\noalign{\global\setlength{\arrayrulewidth}{0.4pt}}
\textcolor{cBBB}{$\bullet$} & \textcolor{cBBB}{$\ttb\ttb\ttb$}
  & \textcolor{cBBB}{$+\,+\,+$} & \textcolor{cBBB}{$0$} & \textcolor{cBBB}{$0$}\\
\textcolor{cBBA}{$\bullet$} & \textcolor{cBBA}{$\ttb\ttb\tta$}
  & \textcolor{cBBA}{$-\,+\,+$} & \textcolor{cBBA}{$1$} & \textcolor{cBBA}{$0$}\\
\textcolor{cBAB}{$\bullet$} & \textcolor{cBAB}{$\ttb\tta\ttb$}
  & \textcolor{cBAB}{$+\,-\,+$} & \textcolor{cBAB}{$1$} & \textcolor{cBAB}{$1$}\\
\textcolor{cABB}{$\bullet$} & \textcolor{cABB}{$\tta\ttb\ttb$}
  & \textcolor{cABB}{$+\,+\,-$} & \textcolor{cABB}{$1$} & \textcolor{cABB}{$2$}\\
\textcolor{cBAA}{$\bullet$} & \textcolor{cBAA}{$\ttb\tta\tta$}
  & \textcolor{cBAA}{$-\,-\,+$} & \textcolor{cBAA}{$2$} & \textcolor{cBAA}{$0$}\\
\textcolor{cAAB}{$\bullet$} & \textcolor{cAAB}{$\tta\tta\ttb$}
  & \textcolor{cAAB}{$+\,-\,-$} & \textcolor{cAAB}{$2$} & \textcolor{cAAB}{$1$}\\
\textcolor{cAAA}{$\bullet$} & \textcolor{cAAA}{$\tta\tta\tta$}
  & \textcolor{cAAA}{$-\,-\,-$} & \textcolor{cAAA}{$3$} & \textcolor{cAAA}{$0$}\\
\end{tabular}
\end{minipage}%
\hfill
\begin{minipage}[c]{0.54\textwidth}
\centering
\begin{tikzpicture}[scale=1.35, font=\small]
  \def\R{2}

  \coordinate (V1)  at (    0   ,  0.4   );
  \coordinate (V2)  at (-0.35   , -0.2   );
  \coordinate (V3)  at ( 0.35   , -0.2   );
  \coordinate (P1a) at (-1.179  , -1.617 );
  \coordinate (P1b) at ( 0.829  ,  1.818 );
  \coordinate (P2a) at (-0.829  ,  1.818 );
  \coordinate (P2b) at ( 1.179  , -1.617 );
  \coordinate (P3a) at (-1.99   , -0.2   );
  \coordinate (P3b) at ( 1.99   , -0.2   );

  \definecolor{sBBB}{RGB}{226, 84, 70}   
  \definecolor{sBBA}{RGB}{246,202, 92}   
  \definecolor{sBAB}{RGB}{122,180, 98}   
  \definecolor{sABB}{RGB}{243,150, 70}   
  \definecolor{sBAA}{RGB}{190,200, 98}   
  \definecolor{sAAB}{RGB}{210,168,122}   
  \definecolor{sAAA}{RGB}{247,217,150}   

  \begin{scope}[shift={(0.05,-0.06)}]
    \fill[black, opacity=0.12] (0,0) circle (2.26);
  \end{scope}

  \fill[crustdark] (0,0) circle (2.24);
  \fill[crustmid]  (0,0) circle (2.15);
  \fill[crusthi]   (0,0) circle (2.07);
  \foreach \a/\rr/\sx/\sy in {
      18/2.13/2.4/1.6, 52/2.11/1.8/2.1, 88/2.14/2.6/1.5,
      121/2.10/2.0/1.9, 158/2.13/2.3/1.6, 193/2.11/1.7/2.2,
      232/2.13/2.5/1.5, 268/2.10/1.9/2.0, 304/2.14/2.4/1.7, 339/2.11/2.1/1.9}{
    \fill[charA, opacity=0.85, rotate around={\a:(0,0)}]
      ([shift={(\a:\rr)}]0,0) ellipse (\sx pt and \sy pt);
  }
  \foreach \a/\rr in {36/2.12, 105/2.11, 176/2.12, 250/2.11, 322/2.12}{
    \fill[charB, opacity=0.7] ([shift={(\a:\rr)}]0,0) circle (1.1pt);
  }

  \begin{scope}
    \clip (0,0) circle (\R);
    \fill[cheese] (0,0) circle (\R);

    \begin{scope}[transparency group, opacity=0.55]
      \fill[sBBB] (V1) -- (P1b)
        arc[start angle=65.5,    end angle=114.5,  radius=\R] -- cycle;
      \fill[sBBA] (V1) -- (V3) -- (P3b)
        arc[start angle=-5.74,   end angle=65.5,   radius=\R] -- cycle;
      \fill[sBAB] (V3) -- (P2b)
        arc[start angle=-53.9,   end angle=-5.74,  radius=\R] -- cycle;
      \fill[sABB] (V3) -- (V2) -- (P1a)
        arc[start angle=-126.1,  end angle=-53.9,  radius=\R] -- cycle;
      \fill[sBAA] (V2) -- (P3a)
        arc[start angle=185.74,  end angle=233.9,  radius=\R] -- cycle;
      \fill[sAAB] (V2) -- (V1) -- (P2a)
        arc[start angle=114.5,   end angle=185.74, radius=\R] -- cycle;
      \fill[sAAA] (V1) -- (V2) -- (V3) -- cycle;
    \end{scope}

    \foreach \x/\y/\r in {
        -1.3/0.9/0.20, 1.1/1.1/0.16, 1.5/-0.3/0.22, -0.5/1.5/0.15,
        0.4/-1.4/0.18, -1.5/-0.6/0.17, 1.3/0.3/0.14, -0.9/-1.2/0.20,
        0.9/-0.9/0.15, -1.7/0.2/0.16, 0.2/1.2/0.13, 1.6/0.7/0.15}{
      \fill[cheeselt, opacity=0.28] (\x,\y) circle (\r);
    }
    \foreach \x/\y/\r in {
        -1.1/1.3/0.14, 1.4/-1.0/0.16, -1.6/-0.2/0.13, 0.7/1.4/0.12,
        1.0/-1.4/0.14, -0.7/-1.5/0.13, 1.7/0.1/0.12}{
      \fill[cheesedk, opacity=0.22] (\x,\y) circle (\r);
    }

    \foreach \k in {1,...,7}{
      \draw[crustdark, opacity=0.05, line width=2pt]
        (0,0) circle ({\R-0.03*\k});
    }
  \end{scope}

  \draw[saucered, line width=1.5pt, line cap=round] (P1a) -- (P1b);
  \draw[saucered, line width=1.5pt, line cap=round] (P2a) -- (P2b);
  \draw[saucered, line width=1.5pt, line cap=round] (P3a) -- (P3b);

  \newcommand{\pepp}[3]{%
    \begin{scope}
      \shade[inner color=pephi, outer color=pepred] (#1,#2) circle (#3);
      \draw[peprim, line width=0.8pt] (#1,#2) circle ({#3-0.012});
      \fill[pepspeck] ([shift={( 0.30*#3, 0.22*#3)}]#1,#2) circle (0.6pt);
      \fill[pepspeck] ([shift={(-0.28*#3, 0.30*#3)}]#1,#2) circle (0.5pt);
      \fill[pepspeck] ([shift={( 0.10*#3,-0.32*#3)}]#1,#2) circle (0.6pt);
      \fill[pepspeck] ([shift={(-0.34*#3,-0.18*#3)}]#1,#2) circle (0.5pt);
      \fill[white, opacity=0.28]
        ([shift={(-0.30*#3,0.34*#3)}]#1,#2) circle ({0.26*#3});
    \end{scope}}
  \pepp{ 0.00}{ 1.74}{0.165}
  \pepp{ 1.51}{ 0.84}{0.150}
  \pepp{ 1.51}{-0.84}{0.150}
  \pepp{ 0.00}{-1.74}{0.165}
  \pepp{-1.51}{-0.84}{0.150}
  \pepp{-1.51}{ 0.84}{0.150}

  \newcommand{\basil}[3]{%
    \begin{scope}[shift={(#1,#2)}, rotate=#3, scale=0.44]
      \fill[basilg]
        (0,0) .. controls (0.55,0.30) and (0.55,0.70) .. (0,1)
              .. controls (-0.55,0.70) and (-0.55,0.30) .. (0,0);
      \draw[basildk, line width=0.5pt] (0,0.04) -- (0,0.96);
      \draw[basilvein, line width=0.35pt] (0,0.32) -- ( 0.26,0.50);
      \draw[basilvein, line width=0.35pt] (0,0.32) -- (-0.26,0.50);
      \draw[basilvein, line width=0.35pt] (0,0.58) -- ( 0.22,0.74);
      \draw[basilvein, line width=0.35pt] (0,0.58) -- (-0.22,0.74);
      \draw[basildk, line width=0.4pt]
        (0,0) .. controls (0.55,0.30) and (0.55,0.70) .. (0,1)
              .. controls (-0.55,0.70) and (-0.55,0.30) .. (0,0);
    \end{scope}}
  \basil{ 0.78}{ 0.96}{ 28}
  \basil{-0.86}{ 0.10}{-58}
  \basil{ 0.30}{-0.92}{200}
  \basil{-0.20}{ 1.02}{-18}

  \fill[seedcol] (V1) circle (1.0pt);
  \fill[seedcol] (V2) circle (1.0pt);
  \fill[seedcol] (V3) circle (1.0pt);

  \draw[crustdark!88!black, line width=0.9pt] (0,0) circle (\R);

  \tikzset{wlabel/.style={
    draw=crustdark!55!black, line width=0.4pt,
    fill=platecol, fill opacity=0.92, text opacity=1,
    inner sep=2.6pt, rounded corners=2.5pt,
    font=\normalsize\bfseries, text=wordcol
  }}
  \node[wlabel] at ( 0   , 1.30)  {$\ttb\ttb\ttb$};
  \node[wlabel] at ( 1.06, 0.56)  {$\ttb\ttb\tta$};
  \node[wlabel] at ( 1.20,-0.66)  {$\ttb\tta\ttb$};
  \node[wlabel] at ( 0   ,-1.16)  {$\tta\ttb\ttb$};
  \node[wlabel] at (-1.20,-0.66)  {$\ttb\tta\tta$};
  \node[wlabel] at (-1.06, 0.56)  {$\tta\tta\ttb$};
  \node[wlabel, font=\scriptsize\bfseries, inner sep=1.6pt]
                  at ( 0   , 0   ) {$\tta\tta\tta$};
\end{tikzpicture}

\medskip
\captionof{figure}{Lazy-caterer dissection for $n=3$.}
\label{fig:pizza-d2}
\end{minipage}

\end{figure}

\section{Class sizes and fixed-cardinality distributions}\label{sec:class-size}

Having counted the equivalence classes in the previous sections, we now study
their sizes.  Using the Young-diagram encoding, we express the cardinality of
each \(\approx_s\)-class through a Gaussian binomial coefficient; specialising
yields the full class-size distribution for \(\sim_2\) and stabilisation results
for the number of classes of any fixed cardinality.

\subsection{General class-size formula for \texorpdfstring{\(s\)}{s}-quasi-binomial equivalence}
\label{subsec:general-class-size-equiv-s}

We compute the cardinalities of the \(\approx_s\)-classes by fixing the recorded
\(s\)-tail and counting the remaining partition inside a rectangle.

We use the Gaussian binomial coefficient
\[
\binom{c}{d}_\qv
=\prod_{i=1}^{d}\frac{1-\qv^{c-d+i}}{1-\qv^{i}},
\]
defined to be \(0\) for \(d<0\) or \(d>c\).  It is the area-generating
polynomial for partitions contained in a \(d\times(c-d)\) rectangle,
\[
\binom{c}{d}_\qv=\sum_{\nu\subseteq (c-d)^{d}} \qv^{|\nu|}.
\]

Recall from \cref{def:s-tail,def:s-residual-area} the quantities
\(\ell=\min(s,|\wrd|_{\tta})\), the last tail entry \(\theta\), and the residual
area \(R_s(\wrd)\) attached to a word \(\wrd\in\{\tta,\ttb\}^n\).

\begin{theorem}
\label{thm:general-class-size-equiv-s}
Let \(\wrd\in\{\tta,\ttb\}^n\).  Then the cardinality \(\#[\wrd]_s\) of the \(\approx_s\)-class of
\(\wrd\) is the coefficient
\[
[\qv^{\,R_s(\wrd)}]\binom{(|\wrd|_{\tta}-\ell)+(|\wrd|_{\ttb}-\theta)}{|\wrd|_{\tta}-\ell}_{\!\qv}.
\]
\end{theorem}

\begin{proof}
The words \(\wu\in[\wrd]_s\) are those whose Young diagram
\(\lambda(\wu)=(r_1,\ldots,r_{|\wrd|_{\tta}})\) has area
\(\binom{\wrd}{\tta\ttb}\) and satisfies \(\mathcal{T}_s(\wu)=\mathcal{T}_s(\wrd)\);
the free parts then satisfy
\(|\wrd|_{\ttb}\geq r_1\geq\cdots\geq r_{|\wrd|_{\tta}-\ell}\geq\theta\).  Setting
\(\nu_i:=r_i-\theta\) for \(1\leq i\leq|\wrd|_{\tta}-\ell\) gives a partition
\(\nu\subseteq(|\wrd|_{\ttb}-\theta)^{\,|\wrd|_{\tta}-\ell}\) of area
\[
|\nu|=\binom{\wrd}{\tta\ttb}-(|\wrd|_{\tta}-\ell)\theta-|\mathcal{T}_s(\wrd)|
=R_s(\wrd).
\]
Conversely, every such \(\nu\) gives free parts \(r_i=\nu_i+\theta\) which,
together with the fixed tail \(\mathcal{T}_s(\wrd)\), reconstruct \(\wu\).  Hence
\(\#[\wrd]_s\) equals the number of such \(\nu\), the stated coefficient; see
\cref{ex:class-size-count}.
\end{proof}

\begin{example}
\label{ex:class-size-count}
\definecolor{raRes}{HTML}{F6BC8C}    
\definecolor{raEdge}{HTML}{3E2E38}   
\definecolor{raResB}{HTML}{C0492A}   
\definecolor{raFrame}{HTML}{E36A48}  
\definecolor{raBox}{HTML}{FAEAD8}    
Take \(s=1\) and \(\wrd=\ttb\tta\ttb\ttb\tta\ttb\tta\ttb\); then \(\ell=\theta=1\)
and \(R_1(\wrd)=4\), so \cref{thm:general-class-size-equiv-s} gives
\(\#[\wrd]_1=3\).

\begin{center}
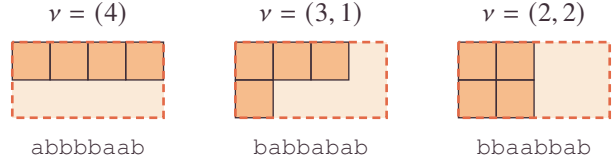

\[
\binom{6}{2}_{\!\qv}=1+\qv+2\qv^{2}+2\qv^{3}
   +\textcolor{raResB}{\mathbf{3}}\,\qv^{4}
   +2\qv^{5}+2\qv^{6}+\qv^{7}+\qv^{8}.
\]
\[
\#[\wrd]_{1}=\bigl[\qv^{4}\bigr]\binom{6}{2}_{\!\qv}=\textcolor{raResB}{\mathbf{3}}.
\]

\medskip
\begin{tikzpicture}[x=0.5cm, y=0.5cm,
   ce/.style={draw=raEdge, line width=0.5pt},
   bx/.style={draw=raFrame, line width=1pt, dash pattern=on 3pt off 2pt},
   pl/.style={font=\small, text=raEdge},
   wl/.style={font=\footnotesize, text=raEdge}]
\foreach \x in {0, 5.9, 11.8}
  \fill[raBox] (\x,1) rectangle (\x+4,-1);
\fill[raRes] (0,0) rectangle (4,1);
\foreach \j in {1,...,4} \draw[ce] (\j-1,0) rectangle (\j,1);
\node[wl] at (2,-1.75) {$\tta\ttb\ttb\ttb\ttb\tta\tta\ttb$};
\node[pl] at (2,1.75) {$\nu=(4)$};
\fill[raRes] (5.9,0) rectangle (8.9,1);
\fill[raRes] (5.9,-1) rectangle (6.9,0);
\foreach \j in {1,2,3} \draw[ce] (5.9+\j-1,0) rectangle (5.9+\j,1);
\draw[ce] (5.9,-1) rectangle (6.9,0);
\node[wl] at (7.9,-1.75) {$\ttb\tta\ttb\ttb\tta\ttb\tta\ttb$};
\node[pl] at (7.9,1.75) {$\nu=(3,1)$};
\fill[raRes] (11.8,0) rectangle (13.8,1);
\fill[raRes] (11.8,-1) rectangle (13.8,0);
\foreach \j in {1,2} \draw[ce] (11.8+\j-1,0) rectangle (11.8+\j,1);
\foreach \j in {1,2} \draw[ce] (11.8+\j-1,-1) rectangle (11.8+\j,0);
\node[wl] at (13.8,-1.75) {$\ttb\ttb\tta\tta\ttb\ttb\tta\ttb$};
\node[pl] at (13.8,1.75) {$\nu=(2,2)$};
\foreach \x in {0, 5.9, 11.8}
  \draw[bx] (\x,1) rectangle (\x+4,-1);
\end{tikzpicture}

\medskip
\captionof{figure}{The class \([\wrd]_1\) has exactly three elements, the three
words shown.}
\label{fig:class-size-three}
\end{center}
\end{example}

\begin{corollary}
\label{cor:full-class-size-distribution-equiv-s}
The class-size distribution polynomial for \(\approx_s\) on \(\{\tta,\ttb\}^n\) is
\[
\sum_{[\wrd]_s\in\{\tta,\ttb\}^n/\approx_s}
z^{\,[\qv^{R_s(\wrd)}]\binom{(|\wrd|_{\tta}-\ell)+(|\wrd|_{\ttb}-\theta)}{|\wrd|_{\tta}-\ell}_{\!\qv}}.
\]
\end{corollary}

\begin{proof}
The summand is independent of the representative \(\wrd\), since \(|\wrd|_{\tta}\),
\(|\wrd|_{\ttb}\), \(\mathcal{T}_s(\wrd)\), and \(R_s(\wrd)\) are fixed on an
\(\approx_s\)-class.  The formula follows from \cref{thm:general-class-size-equiv-s}.
\end{proof}

We record the first two special cases.

\begin{corollary}
\label{cor:class-size-s0}
Let \(\wrd\in\{\tta,\ttb\}^n\).  Then we have:
\begin{enumerate}[label=\textup{(\roman*)}]
\item for \(s=0\),
\[
\#[\wrd]_0=[\qv^{\,\binom{\wrd}{\tta\ttb}}]\binom{n}{|\wrd|_{\tta}}_{\!\qv};
\]
\item for \(s=1\) with \(|\wrd|_{\tta}>0\),
\[
\#[\wrd]_1=[\qv^{\,\binom{\wrd}{\tta\ttb}-|\wrd|_{\tta}\theta}]\binom{n-\theta-1}{|\wrd|_{\tta}-1}_{\!\qv}.
\]
\end{enumerate}
\end{corollary}

\begin{proof}
Specialise \cref{thm:general-class-size-equiv-s}.  For \(s=0\), \(\ell=\theta=0\)
and \(R_0(\wrd)=\binom{\wrd}{\tta\ttb}\), giving (i); for \(s=1\) with
\(|\wrd|_{\tta}>0\), \(\ell=1\) and
\(R_1(\wrd)=\binom{\wrd}{\tta\ttb}-|\wrd|_{\tta}\theta\), giving (ii).
\end{proof}

Complementation in the \(d\times(c-d)\) box exchanges areas \(k\) and \(d(c-d)-k\), so
\(\binom{c}{d}_\qv\) is \emph{palindromic} of degree \(d(c-d)\):
\[
[\qv^{k}]\binom{c}{d}_\qv=[\qv^{\,d(c-d)-k}]\binom{c}{d}_\qv.
\]

\begin{corollary}
\label{cor:fixed-tail-symmetry}
Let \(\wu,\wrd\in\{\tta,\ttb\}^n\) such that \(|\wu|_{\tta}=|\wrd|_{\tta}\),
\(\mathcal{T}_s(\wu)=\mathcal{T}_s(\wrd)\), and
\(R_s(\wu)+R_s(\wrd)=(|\wrd|_{\tta}-\ell)(|\wrd|_{\ttb}-\theta)\).  Then
\(\#[\wu]_s=\#[\wrd]_s\).
\end{corollary}

\begin{proof}
Since \(|\wu|_{\tta}=|\wrd|_{\tta}\) and \(\mathcal{T}_s(\wu)=\mathcal{T}_s(\wrd)\),
\cref{thm:general-class-size-equiv-s} expresses \(\#[\wrd]_s\) and \(\#[\wu]_s\)
as the coefficients of \(\qv^{R_s(\wrd)}\) and \(\qv^{R_s(\wu)}\) in the same
Gaussian binomial, of degree \((|\wrd|_{\tta}-\ell)(|\wrd|_{\ttb}-\theta)\).  By
hypothesis the two exponents sum to this degree, so palindromicity gives
\(\#[\wu]_s=\#[\wrd]_s\).
\end{proof}

\subsection{Class-size formula for \texorpdfstring{\(2\)}{2}-binomial equivalence}
\label{subsec:sim2-class-size-formula}

The cardinalities of \(2\)-binomial equivalence classes were studied by
Lejeune et al.~\cite{LejeuneRigoRosenfeld2020}; for \(k\)-abelian equivalence,
see Karhum\"{a}ki et al.~\cite{KPRW17}.  In the binary fixed-length setting, the
Young-diagram encoding yields a direct coefficient formula, since on
\(\{\tta,\ttb\}^n\) the relation \(\sim_2\) coincides with \(\approx_0\).

\begin{corollary}
\label{thm:sim2-class-size}
Let \(\wrd\in\{\tta,\ttb\}^n\).  Then the cardinality \(\#[\wrd]_{\sim_2}\) of the \(\sim_2\)-class of
\(\wrd\) is the coefficient
\[
[\qv^{\binom{\wrd}{\tta\ttb}}]\binom{n}{|\wrd|_{\tta}}_{\!\qv}.
\]
\end{corollary}

\begin{proof}
Since \(\sim_2\) coincides with \(\approx_0\), this is \cref{cor:class-size-s0}\,(i).
\end{proof}

\begin{corollary}
\label{cor:sim2-class-size-distribution}
The class-size distribution polynomial for \(\sim_2\) on \(\{\tta,\ttb\}^n\) is
\[
\sum_{[\wrd]_{\sim_2}\in\{\tta,\ttb\}^n/{\sim_2}}
z^{\,[\qv^{\binom{\wrd}{\tta\ttb}}]\binom{n}{|\wrd|_{\tta}}_{\!\qv}}.
\]
\end{corollary}

\begin{proof}
The summand is independent of the representative \(\wrd\), since \(|\wrd|_{\tta}\)
and \(\binom{\wrd}{\tta\ttb}\) are fixed on a \(\sim_2\)-class.  The formula follows
from \cref{thm:sim2-class-size}.
\end{proof}

Write \(p\) for the partition function and \(p_{\leq r}\) for the number of
partitions with at most \(r\) parts.  Fix \(K\geq 2\), and set
\[
\rho_K:=\min\{\xi:p(\xi)>K\},\qquad
\varepsilon_K:=
\begin{cases}1,&\text{if \(K\) is a value of \(p\),}\\0,&\text{otherwise.}\end{cases}
\]
For \(2\leq r<\rho_K\), set
\[
c_K(r):=\#\{\xi:p_{\leq r}(\xi)=K\},\qquad
\tau_K(r):=\max\{\xi:p_{\leq r}(\xi)=K\},
\]
the latter being \(-\infty\) when no such \(\xi\) exists.  All these quantities are
independent of \(n\).

\begin{example}
Take \(K=4\).  Then \(\rho_4=4\) (since \(p(3)=3\) and \(p(4)=5\)) and
\(\varepsilon_4=0\) (since \(4\) is not a value of \(p\)).  For \(r=2\),
\(p_{\leq 2}(\xi)=4\) iff \(\xi\in\{6,7\}\), so \(c_4(2)=2\) and \(\tau_4(2)=7\);
for \(r=3\), \(p_{\leq 3}(\xi)=4\) iff \(\xi=4\), so \(c_4(3)=1\) and
\(\tau_4(3)=4\).
\end{example}

For \(r,c\in\N\), set
\[
A_{r,c}(u):=[\qv^{u}]\binom{r+c}{r}_{\qv}.
\]
Equivalently, \(A_{r,c}(u)\) is the number of partitions of \(u\) contained in an
\(r\times c\) rectangle.

\begin{example}
Take \(r=2\) and \(c=3\).  Then
\[
\binom{5}{2}_{\qv}=1+\qv+2\qv^{2}+2\qv^{3}+2\qv^{4}+\qv^{5}+\qv^{6},
\]
so \(A_{2,3}(u)=1,1,2,2,2,1,1\) for \(u=0,\dots,6\); for instance \(A_{2,3}(3)=2\)
counts the partitions \((3)\) and \((2,1)\) in the \(2\times 3\) rectangle.
\end{example}

The following classical properties of \(A_{r,c}\) are used in the proof of
\cref{thm:size-stabilisation-sim2}; see, e.g., \cite{And98} or \cite{Sta12},
the unimodality being originally due to Sylvester.

\begin{lemma}
\label{lem:gaussian-properties}
The function \(A_{r,c}\) has the following properties.
\begin{enumerate}[label=\textup{(\roman*)}]
\item \emph{(Symmetry)} \(A_{r,c}(u)=A_{r,c}(rc-u)\).
\item \emph{(Unimodality)} \(u\mapsto A_{r,c}(u)\) is unimodal on
\(\intv{0}{rc}\).
\item \emph{(Boundary stabilisation)} If \(c\geq u\), then
\(A_{r,c}(u)=p_{\leq r}(u)\); in particular, \(A_{r,c}(u)=p(u)\) when
\(r,c\geq u\).
\end{enumerate}
\end{lemma}

For \(K\geq 1\), let \(N^{\sim_2}_K(n)\) be the number of \(\sim_2\)-classes in
\(\{\tta,\ttb\}^n\) of size \(K\), and for \(K\geq 2\) set
\[
n_0(K):=\max\!\Bigl(2\rho_K-1,\;\max_{2\leq r<\rho_K}\bigl(r+\tau_K(r)+1\bigr)\Bigr).
\]

\begin{theorem}
\label{thm:size-stabilisation-sim2}
Let \(K\geq 2\) and \(n\geq n_0(K)\).  Then
\[
N^{\sim_2}_K(n)=2\varepsilon_K(n-2\rho_K+1)+4\sum_{r=2}^{\rho_K-1}c_K(r).
\]
\end{theorem}

\begin{proof}[Proof of \cref{thm:size-stabilisation-sim2}]
We freely use the symmetry, unimodality, and boundary stabilisations of
\(A_{r,c}(u)=[\qv^{u}]\binom{r+c}{r}_{\qv}\) from \cref{lem:gaussian-properties}.

The threshold \(n\geq n_0(K)\) ensures that, for every
\(r\in\intv{2}{\rho_K-1}\) with \(c_K(r)>0\), all coefficients counted by \(c_K(r)\) lie
strictly before the middle of the Gaussian row: if \(\xi\leq \tau_K(r)\), then
\(n\geq r+\tau_K(r)+1\) gives
\[
r(n-r)-2\xi\;\geq\;r(n-r)-2\tau_K(r)\;>\;0,
\]
so these coefficients and their palindromic partners are distinct.

Since \(n\geq 2\rho_K-1\), there are \(n-2\rho_K+1\) interior rows
\(j\in\intv{\rho_K}{n-\rho_K}\), accounting for the term
\(2\varepsilon_K(n-2\rho_K+1)\).

If \(r,c\geq\rho_K\), then for \(u\in\intv{0}{\rho_K-1}\) we have \(A_{r,c}(u)=p(u)\leq K\),
with equality precisely when \(K\) is a value of \(p\); at \(u=\rho_K\),
\(A_{r,c}(\rho_K)=p(\rho_K)>K\), and unimodality excludes further values \(K\)
before the middle.  Symmetry yields the same contribution at the right end,
so each row with \(r,c\geq\rho_K\) contributes \(2\varepsilon_K\) coefficients
equal to \(K\).

For \(r\in\intv{0}{\rho_K-1}\), the cases \(r=0,1\) contribute nothing since
\(p_{\leq 0}(\xi)\leq 1\) and \(p_{\leq 1}(\xi)=1\).  For \(r\in\intv{2}{\rho_K-1}\) and
\(c=n-r\geq \tau_K(r)+1\), we have \(A_{r,c}(u)=p_{\leq r}(u)\) throughout the
set \(\{u\geq 0:p_{\leq r}(u)=K\}\), all elements of which lie in
\(\intv{0}{\tau_K(r)}\) and hence (by the choice of threshold) strictly below
\(r(n-r)/2\); unimodality plus symmetry then give exactly \(2c_K(r)\)
coefficients equal to \(K\) per row.

Summing over the possible values of \(|\wrd|_{\tta}\), the \(n-2\rho_K+1\) rows with
\(\min(|\wrd|_{\tta},|\wrd|_{\ttb})\geq\rho_K\) each contribute \(2\varepsilon_K\), and the
nontrivial boundary rows pair up as \(\{r,n-r\}\) for \(r\in\intv{2}{\rho_K-1}\), each
pair contributing \(4c_K(r)\).  This yields the stated formula.
\end{proof}

In particular, \(N^{\sim_2}_K(n)\) is eventually linear in \(n\) when \(K\) is a
value of \(p\), and eventually constant otherwise.

\begin{example}\label{ex:fixed-size-small-K}
Consider \(K\in\{2,3,4\}\).  \cref{thm:size-stabilisation-sim2} yields the
eventual counts:
\begin{center}
\renewcommand{\arraystretch}{1.15}
\begin{tabular}{ccc}
\toprule
\(K\) & \(n_0(K)\) & \(N^{\sim_2}_K(n)\) for \(n\geq n_0(K)\) \\
\midrule
2 & 6  & \(2n-2\) \\
3 & 8  & \(2n-2\) \\
4 & 10 & \(12\)   \\
\bottomrule
\end{tabular}
\end{center}
\end{example}

\subsection{Class-size formula for \texorpdfstring{\(3/2\)}{3/2}-binomial equivalence}
\label{subsec:three-halves-class-size-formula}

We compute the \(\sim_{3/2}\)-class sizes by summing the \(\sim_2\)-class sizes
of \cref{thm:sim2-class-size} over blocks of constant \(\operatorname{strip}\).

\begin{proposition}
\label{prop:three-halves-blocks-of-two}
\label{cor:three-halves-class-size}
Let \(\wrd\in\{\tta,\ttb\}^n\) such that \(|\wrd|_{\tta}>0\).  Then the cardinality \(\#[\wrd]_{\sim_{3/2}}\) of the \(\sim_{3/2}\)-class of \(\wrd\) is the sum of coefficients
\[
\sum_{i=0}^{|\wrd|_{\tta}-1}[\qv^{\,|\wrd|_{\tta}\operatorname{strip}(\wrd)+i}]\binom{n}{|\wrd|_{\tta}}_{\!\qv}.
\]
\end{proposition}

\begin{proof}
By definition of \(\sim_{3/2}\), a word \(\wu\) lies in \([\wrd]_{\sim_{3/2}}\) iff
\(|\wu|_{\tta}=|\wrd|_{\tta}\) and
\(\binom{\wu}{\tta\ttb}=|\wrd|_{\tta}\operatorname{strip}(\wrd)+i\) for some
\(i\in\intv{0}{|\wrd|_{\tta}-1}\).  For each such \(i\), \cref{thm:sim2-class-size}
gives \([\qv^{\,|\wrd|_{\tta}\operatorname{strip}(\wrd)+i}]\binom{n}{|\wrd|_{\tta}}_{\!\qv}\)
such words; summing over \(i\) gives the formula.
\end{proof}

\begin{corollary}
\label{cor:three-halves-class-size-distribution}
The class-size distribution polynomial for \(\sim_{3/2}\) on \(\{\tta,\ttb\}^n\) is
\[
z+\sum_{\substack{[\wrd]_{\sim_{3/2}}\in\{\tta,\ttb\}^n/{\sim_{3/2}}\\ |\wrd|_{\tta}>0}}
z^{\,\sum_{i=0}^{|\wrd|_{\tta}-1}[\qv^{\,|\wrd|_{\tta}\operatorname{strip}(\wrd)+i}]\binom{n}{|\wrd|_{\tta}}_{\!\qv}}.
\]
\end{corollary}

\begin{proof}
The term \(z\) accounts for the singleton class \(\{\ttb^n\}\); for the remaining
classes the exponent is given by the class-size formula of \cref{cor:three-halves-class-size},
which is well defined since \(|\wrd|_{\tta}\) and \(\operatorname{strip}(\wrd)\) are fixed on a
\(\sim_{3/2}\)-class.
\end{proof}

\begin{example}
For \(n=4\) and \(|\wrd|_{\tta}=2\), \cref{thm:sim2-class-size} gives the
\(\sim_2\)-class sizes as the coefficients of
\[
\binom{4}{2}_{\qv}=1+\qv+2\qv^{2}+\qv^{3}+\qv^{4},
\]
namely \(1,1,2,1,1\).  Grouping into blocks of two gives the
\(\sim_{3/2}\)-class sizes \(1+1=2\), \(2+1=3\), and \(1\) for
\(\operatorname{strip}(\wrd)=0,1,2\).
\end{example}

For \(K\geq 1\), let \(N^{\sim_{3/2}}_K(n)\) denote the number of
\(\sim_{3/2}\)-classes in \(\{\tta,\ttb\}^n\) of cardinality \(K\).

\begin{proposition}
\label{prop:singleton-three-halves}
Let \(n\in\N\).  Then
\[
N^{\sim_{3/2}}_1(n)=
\begin{cases}
1 & \text{if } n=0,\\
2n & \text{if } n\geq 1.
\end{cases}
\]
\end{proposition}

\begin{proof}
For \(n=0\), the empty word is the unique class.  Assume \(n\geq 1\).
The class \(\{\ttb^{n}\}\) is a singleton.  Every class with \(|\wrd|_{\tta}=1\)
is also a singleton: then \(\operatorname{strip}(\wrd)=\binom{\wrd}{\tta\ttb}\)
fixes the position of the single \(\tta\), giving \(n\) further singletons.  The \(n-1\)
words \(\tta^{r}\ttb^{\,n-r}\), \(r\in\intv{2}{n}\), are also singletons, because each
has the maximal possible value \(\binom{\wrd}{\tta\ttb}=|\wrd|_{\tta}|\wrd|_{\ttb}\)
among words with the same value of \(|\wrd|_{\tta}\).

No other class is a singleton: if \(|\wrd|_{\tta}\geq 2\) and
\(\operatorname{strip}(\wrd)<|\wrd|_{\ttb}\), the distinct words
\[
\wu:=\ttb^{\,|\wrd|_{\ttb}-\operatorname{strip}(\wrd)}\tta^{\,|\wrd|_{\tta}}\ttb^{\,\operatorname{strip}(\wrd)},
\qquad
\wv:=\ttb^{\,|\wrd|_{\ttb}-\operatorname{strip}(\wrd)-1}\tta\ttb\,\tta^{\,|\wrd|_{\tta}-1}\ttb^{\,\operatorname{strip}(\wrd)}
\]
satisfy \(|\wu|_{\tta}=|\wv|_{\tta}\) and \(\binom{\wu}{\tta\ttb}=|\wrd|_{\tta}\operatorname{strip}(\wrd)\),
\(\binom{\wv}{\tta\ttb}=|\wrd|_{\tta}\operatorname{strip}(\wrd)+1\), so
\(\operatorname{strip}(\wu)=\operatorname{strip}(\wv)=\operatorname{strip}(\wrd)\) and \(\wu\sim_{3/2}\wv\) by
\cref{thm:32-tfae}.

Hence the number of singleton classes is \(1+n+(n-1)=2n\).
\end{proof}

For \(r\geq 2\), \(\alpha\in\N\), and \(\beta\in\N_{>0}\), define the block sums
\[
L_r(\alpha):=\sum_{\xi=\alpha r}^{(\alpha+1)r-1}p_{\leq r}(\xi),
\qquad
R_r(\beta):=\sum_{\xi=(\beta-1)r+1}^{\beta r}p_{\leq r}(\xi).
\]
Each adds \(p_{\leq r}\) over a block of \(r\) consecutive integers: the blocks
\(\intv{\alpha r}{(\alpha+1)r-1}\) partition \(\N\), while the blocks
\(\intv{(\beta-1)r+1}{\beta r}\) are this partition shifted up by one.

\begin{example}
Assume that \(r=2\), so \(p_{\leq 2}(\xi)=1,1,2,2,3,3,\dots\).
The blocks \(\{0,1\},\{2,3\},\{4,5\},\dots\) give \(L_2(0)=2\), \(L_2(1)=4\),
\(L_2(2)=6,\dots\), while the shifted blocks \(\{1,2\},\{3,4\},\{5,6\},\dots\)
give \(R_2(1)=3\), \(R_2(2)=5\), \(R_2(3)=7,\dots\).  For large \(n\) these
are the \(\sim_{3/2}\)-class sizes with \(|\wrd|_{\tta}=2\) at the lower (from
\(L_2\), even) and upper (from \(R_2\), odd) ends of the strip range.
\end{example}

Set \(n_1(K):=K^{2}+K-1\).

\begin{theorem}
\label{thm:size-stabilisation-three-halves}
Let \(K\geq 2\) and \(n\geq n_1(K)\).  Then
\[
N^{\sim_{3/2}}_K(n)=\sum_{r=2}^{K}\bigl(\#\{\alpha\in\N:L_r(\alpha)=K\}
+\#\{\beta\in\N_{>0}:R_r(\beta)=K\}\bigr).
\]
\end{theorem}

\begin{proof}
Fix \(K\geq 2\), and let \(\wrd\in\{\tta,\ttb\}^n\) represent a non-singleton
\(\sim_{3/2}\)-class.  By \cref{prop:three-halves-blocks-of-two} this class has
size
\[
\sum_{i=0}^{|\wrd|_{\tta}-1}[\qv^{|\wrd|_{\tta}\operatorname{strip}(\wrd)+i}]\binom{n}{|\wrd|_{\tta}}_{\qv},
\]
and by \cref{prop:singleton-three-halves} the singleton cases
\(|\wrd|_{\tta}\leq 1\) and \(\operatorname{strip}(\wrd)=|\wrd|_{\ttb}\) may be
discarded, so \(|\wrd|_{\tta}\geq 2\) and
\(\operatorname{strip}(\wrd)\in\intv{0}{|\wrd|_{\ttb}-1}\); since the block carries
\(|\wrd|_{\tta}\) positive coefficients, a class of size \(K\) has
\(|\wrd|_{\tta}\leq K\).

Adding \(d\) to the first part and subtracting \(d\) from the last, for
\(d\in\intv{0}{\min\{\operatorname{strip}(\wrd),|\wrd|_{\ttb}-\operatorname{strip}(\wrd)\}}\),
turns the partition with all \(|\wrd|_{\tta}\) parts equal to
\(\operatorname{strip}(\wrd)\) into
\(\min\{\operatorname{strip}(\wrd),|\wrd|_{\ttb}-\operatorname{strip}(\wrd)\}+1\) distinct
partitions of area \(|\wrd|_{\tta}\operatorname{strip}(\wrd)\) in the
\(|\wrd|_{\tta}\times|\wrd|_{\ttb}\) rectangle, so
\[
[\qv^{|\wrd|_{\tta}\operatorname{strip}(\wrd)}]\binom{n}{|\wrd|_{\tta}}_{\qv}
\geq\min\{\operatorname{strip}(\wrd),|\wrd|_{\ttb}-\operatorname{strip}(\wrd)\}+1.
\]
Hence \(\min\{\operatorname{strip}(\wrd),|\wrd|_{\ttb}-\operatorname{strip}(\wrd)\}\leq K-1\),
i.e.\ \(\operatorname{strip}(\wrd)\leq K-1\) or
\(|\wrd|_{\ttb}-\operatorname{strip}(\wrd)\leq K-1\); these alternatives are disjoint,
since \(|\wrd|_{\ttb}\geq K^{2}-1>2K-2\).

The threshold \(n\geq n_1(K)=K^{2}+K-1\) places every such block inside the
stable range \(u\leq|\wrd|_{\ttb}\), where
\([\qv^{u}]\binom{n}{|\wrd|_{\tta}}_{\qv}=p_{\leq|\wrd|_{\tta}}(u)\): in the left
case the right endpoint satisfies
\((\operatorname{strip}(\wrd)+1)|\wrd|_{\tta}-1\leq K^{2}-1\leq|\wrd|_{\ttb}\), while
in the right case, with \(\beta:=|\wrd|_{\ttb}-\operatorname{strip}(\wrd)\leq K-1\),
Gaussian palindromy (\cref{cor:fixed-tail-symmetry}) reflects the block onto
\(\intv{(\beta-1)|\wrd|_{\tta}+1}{\beta|\wrd|_{\tta}}\), where
\(\beta|\wrd|_{\tta}\leq K^{2}-K\leq|\wrd|_{\ttb}\).  Hence the class size is
\(L_{|\wrd|_{\tta}}(\operatorname{strip}(\wrd))\) when
\(\operatorname{strip}(\wrd)\leq K-1\) and \(R_{|\wrd|_{\tta}}(\beta)\)
when \(\beta\leq K-1\).  As the value of \(|\wrd|_{\tta}\) ranges over \(r\in\intv{2}{K}\), the
pre-images \(L_r^{-1}(K)\) and \(R_r^{-1}(K)\) count the size-\(K\) classes, and
summing gives the formula.
\end{proof}

Since the pre-images \(L_r^{-1}(K)\) and \(R_r^{-1}(K)\) lie in \(\intv{0}{K-1}\)
and \(\intv{1}{K-1}\), the right-hand side is independent of \(n\); hence
\(N^{\sim_{3/2}}_K(n)\) is eventually constant in \(n\) for each \(K\geq 2\).

\subsection{Stabilisation of \texorpdfstring{\(s\)}{s}-quasi-binomial equivalence and singleton classes}
\label{subsec:stabilisation-singletons}

Singleton \(\approx_s\)-classes were studied from a formal-language viewpoint by
Lejeune et al.~\cite{LejeuneRigoRosenfeld2020}; we count them.  By
\cref{cor:equiv-s-stabilisation}, \(\approx_s\) is the equality relation when
\(n\le s+3\), where the count is immediate; the remaining case \(n\ge s+4\)
reduces to the following fact.

\begin{lemma}
\label{lem:unit-coefficients-gaussian}
The value \(A_{r,c}(u)\) equals \(1\) for all \(u\in\intv{0}{rc}\) when
\(\min(r,c)\leq 1\), and only for \(u\in\{0,1,rc-1,rc\}\) when \(\min(r,c)\geq 2\).
\end{lemma}

\begin{proof}
The number \(A_{r,c}(u)\) counts the partitions of \(u\) contained in the
\(r\times c\) rectangle.  If \(r=0\) or \(c=0\), only the empty partition fits,
so only \(u=0\) occurs.  If \(r=1\), the partitions \((u)\), \(u\in\intv{0}{c}\),
give one partition for each area; the case \(c=1\) is symmetric.

Assume now that \(r,c\geq 2\).  Only \(\emptyset\) and \((1)\) have areas \(0\)
and \(1\), respectively.  By the symmetry of \(A_{r,c}\)
(\cref{lem:gaussian-properties}), the same holds for the complementary areas
\(rc\) and \(rc-1\).  For \(u\in\intv{2}{rc-2}\), symmetry lets us assume
\(u\leq rc/2\).  Since the two partitions \((2)\) and \((1,1)\) both fit in the
rectangle, \(A_{r,c}(2)\geq 2\); unimodality (\cref{lem:gaussian-properties})
then gives \(A_{r,c}(u)\geq 2\).  Hence \(A_{r,c}(u)=1\) exactly for
\(u\in\{0,1,rc-1,rc\}\).
\end{proof}

For \(s,n\in\N\), put
\[
S_{s}(n):=\#\bigl\{[\wrd]_s\in\{\tta,\ttb\}^n/\approx_s:\#[\wrd]_s=1\bigr\},
\]
and let \(\mathcal{S}_s(x):=\sum_{n\geq 0}S_{s}(n)\,x^{n}\) be its generating function.

\begin{theorem}
\label{thm:singleton-classes-equiv-s}
Let \(s\in\N\).  Then
\[
\mathcal{S}_s(x)
=\sum_{j=0}^{s}\frac{x^{j}}{(1-x)^{j+1}}
+\frac{x^{s+1}(1+x)(1+2x^{2})}{(1-x)^{s+2}}.
\]
\end{theorem}

\begin{proof}
By \cref{thm:general-class-size-equiv-s} and \cref{lem:unit-coefficients-gaussian},
a class is a singleton iff its size
\(A_{|\wrd|_{\tta}-\ell,\,|\wrd|_{\ttb}-\theta}(R_s(\wrd))\) equals \(1\), where
\(\ell=\min(s,|\wrd|_{\tta})\).

For \(|\wrd|_{\tta}\leq s\) we have \(\ell=|\wrd|_{\tta}\), so the first index
is \(0\) and every such class is a singleton, contributing
\[
\sum_{j=0}^{s}\sum_{n\geq 0}\binom{n}{j}\,x^{n}
=\sum_{j=0}^{s}\frac{x^{j}}{(1-x)^{j+1}}.
\]

For \(|\wrd|_{\tta}>s\) we have \(\ell=s\), and by \cref{lem:unit-coefficients-gaussian}
a class is a singleton iff its residual area is a unit coefficient of the
residual rectangle's Gaussian binomial---every area when a side of the
rectangle is at most \(1\), and the four extreme areas otherwise.  The \(s\)-tail is a weakly increasing tuple of non-negative integers, empty when
\(s=0\); for \(s=0\) it contributes the factor
\(1\), while for \(s\geq 1\) exactly \(\binom{\theta+s-1}{s-1}\) of these have
last entry \(\theta\), so they contribute
\(\sum_{\theta\geq 0}\binom{\theta+s-1}{s-1}x^{\theta}=1/(1-x)^{s}\).  Either way
the factor is \(1/(1-x)^{s}\).
Weighting each residual rectangle (which has at least one row) by \(x\) raised
to the sum of its two side lengths and by its number of unit coefficients gives
\(x(1+x)(1+2x^{2})/(1-x)^{2}\).  Since \(n\) equals \(s\) plus \(\theta\) plus
that sum, this case contributes
\[
x^{s}\cdot\frac{1}{(1-x)^{s}}\cdot\frac{x(1+x)(1+2x^{2})}{(1-x)^{2}}
=\frac{x^{s+1}(1+x)(1+2x^{2})}{(1-x)^{s+2}}.
\]

Summing the two cases gives the stated generating function.
\end{proof}

Extracting coefficients gives the closed form
\[
S_{s}(n)=\sum_{j=0}^{s}\binom{n}{j}+\binom{n}{s+1}+\binom{n-1}{s+1}+2\binom{n-2}{s+1}+2\binom{n-3}{s+1}.
\]
In particular, \(S_{s}(n)=2^{n}\) for \(n\leq s+3\), since then \(\approx_s\) is
the equality relation by \cref{cor:equiv-s-stabilisation}.

\begin{example}
Specialising \cref{thm:singleton-classes-equiv-s} to \(s=0,1,2\) gives, for
\(n\geq s+4\),
\[
S_{0}(n)=6n-10,\qquad S_{1}(n)=3n^{2}-13n+20,\qquad S_{2}(n)=n^{3}-8n^{2}+27n-28,
\]
while \(S_{s}(n)=2^{n}\) for \(n\leq s+3\), with which each polynomial agrees at
\(n=s+3\).
\end{example}

\begin{corollary}
\label{cor:first-collision}
Let \(s\in\N\).  Then \(\approx_s\) has exactly one non-singleton class on
\(\{\tta,\ttb\}^{s+4}\), namely
\[
\left\{\tta\ttb^{2}\tta^{s+1},\ \ttb\tta^{2}\ttb\tta^{s}\right\}.
\]
\end{corollary}

\begin{proof}
By \cref{thm:higher-tail-schlaefli-dictionary},
\(\#(\{\tta,\ttb\}^{s+4}/\approx_s)=\sum_{r=0}^{s+3}\binom{s+4}{r}=2^{s+4}-1\);
since there are \(2^{s+4}\) words, exactly one class is non-singleton, of size
\(2\).  The two displayed words have Young diagrams \((2,0,\ldots,0)\) and
\((1,1,0,\ldots,0)\), sharing \(|\wrd|_{\tta}=s+2\), area \(2\), and \(s\)-tail
\((0,\ldots,0)\); they are therefore \(s\)-quasi-binomially equivalent and constitute that
class.
\end{proof}

At \(s=0\) these words are \(\tta\ttb\ttb\tta\) and \(\ttb\tta\tta\ttb\), the
first nontrivial \(\sim_2\)-collision, occurring at \(n=4\).

Asymptotically, for fixed \(s\), \cref{thm:singleton-classes-equiv-s}
gives \(S_{s}(n)=\tfrac{6}{(s+1)!}\,n^{s+1}+O_{s}(n^{s})\), while
\cref{thm:higher-tail-schlaefli-dictionary} gives
\(\#(\{\tta,\ttb\}^{n}/\approx_s)=\tfrac{1}{(s+3)!}\,n^{s+3}+O_{s}(n^{s+2})\);
hence
\[
\frac{S_{s}(n)}{\#(\{\tta,\ttb\}^{n}/\approx_s)}=O_{s}(n^{-2}),
\]
so singleton classes form an asymptotically negligible proportion of all
\(\approx_s\)-classes.


\makeatletter
\let\orig@thebibliography\thebibliography
\renewcommand{\thebibliography}[1]{%
    \orig@thebibliography{#1}%
    \setlength{\leftmargin}{5.5em}%
    \setlength{\itemindent}{-2em}%
    \setlength{\labelwidth}{2.5em}%
    \setlength{\labelsep}{0.5em}%
    \setlength{\rightmargin}{3em}%
    \footnotesize
    \renewcommand{\makelabel}[1]{\hfill\textnormal{\scriptsize ##1}}%
}
\makeatother

\hypersetup{linkcolor=backrefcolor}

\end{document}